\documentclass[twoside,accepted]{article}

\usepackage[ruled, linesnumbered]{algorithm2e} %
\usepackage{xcolor} %
\usepackage{tikz}
\usetikzlibrary{fit,calc}

\usepackage{aistats2025}
\usepackage{url}            
\usepackage{booktabs}       
\usepackage{amsfonts}       
\usepackage{microtype}      

\usepackage[linktocpage=true, colorlinks=true, allcolors=black]{hyperref}     

\usepackage{nicefrac} 
\usepackage{subcaption}
\usepackage{amsmath}
\usepackage{amsthm}
\usepackage{amssymb}
\usepackage{thmtools}
\usepackage{paralist}
\usepackage{comment}
\usepackage{dsfont}
\usepackage{xspace}

\usepackage{bbm}
\usepackage{float}
\usepackage{balance}
\usepackage{stmaryrd}

\usepackage{multicol}
\usepackage[export]{adjustbox}
\usepackage{multirow}
\usepackage{tabularx}
\usepackage{mathtools}
\usepackage{mdframed}

\graphicspath{{../}}

\usepackage[shortlabels]{enumitem}
\setitemize{noitemsep,topsep=0pt,parsep=0pt,partopsep=0pt}
\usepackage[round]{natbib}

\def\Rfwt{{R}}

\def\asc{{\textnormal{\texttt{ASC}}}\xspace}

\def\afw{{\textnormal{\texttt{AFW}}}\xspace}

\def\bpfw{{\textnormal{\texttt{BPFW}}}\xspace}

\def\fw{{\textnormal{\texttt{FW}}}\xspace}

\def\irr{{\textnormal{\texttt{IRR}}}\xspace}

\def\fwu{{\textnormal{\texttt{FWU}}}\xspace}
\def\caa{{\textnormal{\texttt{CA}}}\xspace}

\def\ccu{{\textnormal{\texttt{CCU}}}\xspace}

\def\pm{{\textnormal{\texttt{PM}}}\xspace}

\def\lmo{{\textnormal{\texttt{LMO}}}\xspace}

\def\aa{{\mathbf{a}}}

\def\cc{{\mathbf{c}}}
\def\dd{{\mathbf{d}}}
\def\ee{{\mathbf{e}}}

\def\pp{{\mathbf{p}}}
\def\rr{{\mathbf{r}}}
\def\ss{{\mathbf{s}}}

\def\vv{{\mathbf{v}}}
\def\ww{{\mathbf{w}}}
\def\xx{{\mathbf{x}}}
\def\yy{{\mathbf{y}}}
\def\zz{{\mathbf{z}}}

\def\aalpha{\boldsymbol{\alpha}}
\def\bbeta{\boldsymbol{\beta}}
\def\ggamma{{\boldsymbol{\gamma}}}

\def\llambda{{\boldsymbol{\lambda}}}
\def\mmu{{\boldsymbol{\mu}}}

\newcommand{\N}{\mathbb{N}}

\newcommand{\R}{\mathbb{R}}

\newcommand\cC{{\ensuremath{\mathcal{C}}}\xspace}
\newcommand\cD{{\ensuremath{\mathcal{D}}}\xspace}

\newcommand\cO{{\ensuremath{\mathcal{O}}}\xspace}
\newcommand\cP{{\ensuremath{\mathcal{P}}}\xspace}

\newcommand\cS{{\ensuremath{\mathcal{S}}}\xspace}

\newcommand\cT{{\ensuremath{\mathcal{T}}}\xspace}

\newcommand\cV{{\ensuremath{\mathcal{V}}}\xspace}

\newcommand\cX{{\ensuremath{\mathcal{X}}}\xspace}

\DeclareMathOperator{\relint}{rel.int}
\DeclareMathOperator{\aff}{aff}
\DeclareMathOperator{\vertices}{vert}

\DeclareMathOperator{\dist}{dist}
\DeclareMathOperator{\faces}{faces}

\DeclareMathOperator{\argmax}{argmax}
\DeclareMathOperator{\argmin}{argmin}
\DeclareMathOperator{\conv}{conv}

\DeclareMathOperator{\supp}{supp}

\newcommand{\zeros}{\ensuremath{\mathbf{0}}}
\newcommand{\ones}{\mathbf{1}}

\newcommand{\norm}[1]{\| #1 \|}

\theoremstyle{plain} \numberwithin{equation}{section}
\newtheorem{theorem}{Theorem}[section]
\numberwithin{theorem}{section}
\newtheorem{lemma}[theorem]{Lemma}
\newtheorem{corollary}[theorem]{Corollary}
\newtheorem{proposition}[theorem]{Proposition}

\theoremstyle{definition}
\newtheorem{definition}[theorem]{Definition}

\theoremstyle{plain}
\newtheorem{assumption}{Assumption}

\newcommand{\hrulealg}[0]{\vspace{1mm} \hrule \vspace{1mm}}

\begin{document}

%

%

\runningtitle{The Pivoting Framework: Frank-Wolfe Algorithms with Active Set Size Control}

\twocolumn[

\aistatstitle{The Pivoting Framework:
Frank-Wolfe Algorithms \\ with Active Set Size Control}

\aistatsauthor{ Elias Wirth \And Mathieu Besan\c con \And  Sebastian Pokutta}

\aistatsaddress{ TU Berlin \And  Université Grenoble Alpes\\Inria, CNRS, LIG \And TU Berlin \\ Zuse Institute Berlin } ]



\begin{abstract}
We propose the pivoting meta algorithm (\pm{}) to enhance optimization algorithms that generate iterates as convex combinations of vertices of a feasible region $\cC\subseteq \R^n$, including Frank-Wolfe (\fw) variants. \pm guarantees that the active set (the set of vertices in the convex combination) of the modified algorithm remains as small as $\dim(\cC)+1$ as stipulated by Carathéodory's theorem. \pm{} achieves this by reformulating the active set expansion task into an equivalent linear program, which can be efficiently solved using a single pivot step akin to the primal simplex algorithm; the convergence rate of the original algorithms are maintained. Furthermore, we establish the connection between \pm{} and active set identification, in particular showing under mild assumptions that \pm{} applied to the away-step Frank-Wolfe algorithm (\afw{}) or the blended pairwise Frank-Wolfe algorithm (\bpfw{}) bounds the active set size by the dimension of the optimal face plus $1$. We provide numerical experiments to illustrate practicality and efficacy on active set size reduction.
\end{abstract}

\section{Introduction}\label{sec:introduction}
We study constrained convex optimization problems
\begin{equation}\label{eq:opt}\tag{OPT}
    \min_{\xx\in\cC}f(\xx),
\end{equation}
where $\cC\subseteq\R^n$ is a compact convex set with vertex set $\cV = \vertices(\cC)$ and $f\colon \cC \to \R$ is a convex and smooth function.
%
When projecting onto $\cC$ is computationally challenging, we can address \eqref{eq:opt} using the projection-free \emph{Frank-Wolfe algorithm} (\fw{}) \citep{frank1956algorithm}, a.k.a. the \emph{conditional gradients algorithm} \citep{levitin1966constrained}. The \fw{} algorithm, presented in Algorithm~\ref{alg:fw}, only requires first-order access to the function $f$ and a \emph{linear minimization oracle} (\lmo{}) for the feasible region $\cC$. The \lmo{} returns a point in $\argmin_{\xx\in\cC}\langle\cc,\xx\rangle$ when given $\cc\in\R^n$.
\fw{} possesses favorable attributes such as ease of implementation, affine invariance \citep{lacoste2013affine, kerdreux2021affine, pena2023affine}, and its iterates are sparse convex combinations of vertices
of $\cC$. At each iteration, the \fw{} algorithm calls the \lmo{} to obtain a new \fw{} vertex $\vv^{(t)}\in\cV$. As presented in Algorithm~\ref{alg:fw}, the current iterate $\xx^{(t)}$ is updated with line-search $\eta^{(t)} = \argmin_{\eta \in [0, 1]} f(\xx^{(t)} + \eta (\vv^{(t)}-\xx^{(t)}))$ in Line~\ref{line:fw_step} to obtain $\xx^{(t+1)}= \xx^{(t)} + \eta^{(t)} (\vv^{(t)} - \xx^{(t)})$. Alternative step-size rules exist, including short step, adaptive \citep{pedregosa2018step,P2023}, and open-loop variants $\eta^{(t)} = \frac{\ell}{t+\ell}$ for some $\ell\in\N_{>0}$ \citep{dunn1978conditional, wirth2023acceleration, wirth2023accelerated}.

One drawback of the vanilla \fw{} algorithm is its sublinear convergence rate when potentially higher rates are possible, e.g., when the feasible region is a polytope, the objective is strongly convex, and the optimizer lies in the relative interior of a face of $\cC$ (see e.g.,\citet{wolfe1970convergence, bach2021effectiveness, wirth2023acceleration}). Consequently, several variants have been proposed to achieve linear convergence rates in such scenarios (see e.g., \citet{holloway1974extension,guelat1986some,lacoste2015global,garber2016linear,tsuji2022pairwise}). Most of these variants store the current iterate $\xx^{(t)} = \sum_{\ss \in \cS^{(t)}} \alpha_\ss^{(t)}\ss$ as a convex combination of vertices, where $\aalpha^{(t)}\in \Delta_{|\cV|}$, and $\cS^{(t)} = \{\ss\in\cV \mid \alpha_{\ss}^{(t)} > 0 \}$ denote the \emph{weight vector} and \emph{active set}, respectively.
Explicit access to a convex decomposition of $\xx^{(t)}$ enables re-optimization over the active set or aggressive removal of weight from specific vertices and is crucial for achieving linear convergence rates.
Moreover, variants often enhance the sparsity-inducing properties of vanilla \fw{}, which is advantageous in, e.g., deriving bounds for the approximate Carathéodory theorem \citep{combettes2023revisiting}, approximate vanishing ideal computations \citep{wirth2022conditional, wirth2023approximate}, data-driven identification of nonlinear dynamics \citep{carderera2021cindy}, deep neural network training \citep{pokutta2020deep, macdonald2022interpretable}, kernel herding \citep{bach2012equivalence, tsuji2022pairwise, wirth2023acceleration}, robust matrix recovery \citep{mu2016scalable}, and tensor completion \citep{guo2017efficient,bugg2022nonnegative}. 
%

%

Maintaining the active set introduces computational overhead, especially in high-dimensional or dense vertex scenarios due to memory constraints. To improve efficiency, several methods aim to reduce the active set size. One approach alternates between adding vertices and performing correction steps, either fully or partially re-optimizing the set. Examples include the \emph{fully-corrective Frank-Wolfe algorithm} \citep{holloway1974extension, rao2015forward} and the \emph{Blended Frank-Wolfe algorithm (BFW)} \citep{braun2019blended}, also known as \emph{Blended Conditional Gradients (BCG)}. Another approach uses \emph{drop steps} to prune vertices from the active set. Key algorithms include the \emph{away-step Frank-Wolfe algorithm} (\afw{}) \citep{wolfe1970convergence, guelat1986some, lacoste2015global}, detailed in Algorithm~\ref{alg:afw}, the \emph{pairwise Frank-Wolfe algorithm} \citep{lacoste2015global}, the \emph{decomposition-invariant Frank-Wolfe algorithm} \citep{garber2016linear}, and the \emph{Blended Pairwise Frank-Wolfe algorithm (BPFW)} \citep{tsuji2022pairwise}, outlined in Algorithm~\ref{alg:bpfw}, also called \emph{Blended Pairwise Conditional Gradients (BPCG)}.

In \fw{} variants, the size of the active set is typically only bounded solely by the number of iterations performed and the number of vertices in the feasible region, that is, $|\cS^{(t)}|\leq \min\{t + 1, |\cV|\}$. Notably, this bound does not depend on the ambient dimension $n$. This observation is somewhat surprising considering the well-known \emph{Carathéodory theorem}, which guarantees that any vector $\xx\in\cC$ can always be expressed as a convex combination of at most $n + 1$ vertices of $\cC$.
To the best of our knowledge, the only approach proposed to bound the number of vertices to match that of the Carathéodory theorem is the \emph{incremental representation reduction algorithm} (\irr{}) algorithm of \citet{beck2017linearly},
to which we compare our approach in Subsection \ref{sec:irr}.

\begin{theorem}[\citealp{caratheodory1907variabilitatsbereich}]\label{thm:caratheodory}
Let $\cC\subseteq \R^n$ be a compact convex set. Then, any $\xx\in \cC$ can be represented as a convex combination of at most $n + 1$ vertices of $\cC$.
\end{theorem}
In settings where the number of vertices is significantly larger than the dimension, such as, for example, in the convex hull membership problem \citep{filippozzi2023first}, a dimension-dependent upper bound on the size of the active set is preferable to a bound based solely on the number of vertices of $\cC$. 

\begin{algorithm}[t]
\SetKwInput{Input}{Input} \SetKwInput{Output}{Output}
\SetKwComment{Comment}{$\triangleright$\ }{}
\caption{Frank-Wolfe algorithm (\fw{}) with line-search}\label{alg:fw}
  \Input{$\xx^{(0)}\in \cV$.}
  \Output{$\xx^{(T)} \in \cC$.}
  \hrulealg
  \For{$t= 0, 1, \ldots, T-1 $}{
        {$\vv^{(t)} \gets \argmin_{\vv \in \cV} \langle\nabla f(\xx^{(t)}), \vv - \xx^{(t)}\rangle$}\\
        {$\eta^{(t)} \gets \argmin_{\eta \in [0, 1]} f(\xx^{(t)} + \eta (\vv^{(t)}-\xx^{(t)}))$ }\\
        {$\xx^{(t+1)} \gets \xx^{(t)} + \eta^{(t)} ( \vv^{(t)} - \xx^{(t)})$}\label{line:fw_step}}
\end{algorithm}

\subsection{Contributions}\label{sec:contributions}
In this paper, we address the existing gap in the literature by introducing the \emph{pivoting meta algorithm} (\pm) presented in Algorithm~\ref{alg:pm}. Our contributions can be summarized as follows:

\paragraph{Active-Set Reduction} First, \pm is designed to enhance a family of optimization algorithms, including various existing variants of the Frank-Wolfe algorithm (\fw{}). Our main result, Theorem~\ref{thm:pm}, demonstrates the key advantage of \pm: \pm applied to certain optimization algorithms ensures that the cardinality of the active set remains bounded by $n + 1$ while preserving the convergence rate guarantees of the original algorithm. To achieve this, \pm transforms the task of adding a new vertex $\vv^{(t)}\in\cV\setminus \cS^{(t)}$ to an active set $\cS^{(t)}$ into an equivalent linear programming problem. This problem can be solved using a single pivot step, similar to the primal simplex algorithm \citep{bertsimas1997introduction}.
However, this modification introduces the additional computational complexity of solving an $(n+2)\times (n+2)$ linear system in iterations performing an \fw step,  typically via an LU-decomposition. We highlight however, how this computational burden is alleviated by the fact that only sparse rank-one updates are performed on the system to solve, making it amenable to efficient factorization updates as performed in modern simplex solvers.

\paragraph{Active set identification} Second, we establish a connection between \pm and \emph{active set identification}, the process of identifying the face containing a solution of \eqref{eq:opt} and not to be confused with the active set of \fw algorithms. When the feasible region is a polytope and the minimizers lie in the relative interior of a face $\cC^*$ of $\cC$,
\citet{bomze2020active} provided sufficient conditions that guarantee the existence of an iteration $\Rfwt\in\{0,1, \ldots,T\}$, where $T$ is the number of iterations \afw{} is run for, s.t. for all $t\in\{\Rfwt, \Rfwt+1, \ldots, T\}$, the active set $\cS^{(t)}$ produced by \afw{} is contained in the optimal face $\cC^*$, implying that $\xx^{(t)}\in\cC^*$. In the same setting, we prove that applying \pm to \afw{} guarantees an active set size of at most $\dim(\cC^*) + 1$ for all iterations $t\in\{\Rfwt, \Rfwt+1, \ldots, T\}$.\footnote{The result by \citet{bomze2020active} is not limited to \afw{}. In Appendix~\ref{appendix:bpfw_asi}, we derive a similar result for \bpfw{}.}
Since our result holds for a more general setting, we further improve the upper bound on the size of the active set of $n + 1$ to $\dim(\cC) + 1$ for any compact and convex feasible region $\cC$.

\paragraph{Numerical Experiments} Finally, we provide an algorithmic implementation by applying \pm to \afw{} and the blended pairwise Frank-Wolfe algorithm (\bpfw{}) and comparing the method to \afw{} and \bpfw{}.

Some numerical experiments, proofs, and the discussion of implementation details have been relegated to the supplementary material due to space constraints.

%

\section{Preliminaries}\label{sec:preliminaries}

For $n\in\N$, let $[n]:=\{1,\ldots,n\}$.
Vectors are denoted in bold. Given $\xx,\yy\in\R^n$, let $\xx\geq \yy$ denote that $x_i\geq y_i$ for all $i\in [n]$. Given $\xx \in \R^{n}$, let $\tilde{\xx} \in \R^{n+2}$ be defined  as $(\xx^\intercal, 0, 1)^\intercal\in\R^{n+2}$. We denote the $i$th unit vector of dimension $n$ by $\ee_i \in \R^n$. Given a vector $\xx\in \R^{n}$, denote its support by $\supp(\xx):= \{i\in[n] \mid x_i \neq 0\}$.
Given a matrix $M\in \R^{m \times n}$, we refer to the $(i,j)$th entry of $M$ as $M_{i,j}$ and the $i$th row and column of $M$ as $M_{i, :}$ and $M_{:, i}$, respectively. Furthermore, let $M_{:i, :j}$ denote the restriction of matrix $M$ to rows $1, \ldots, i$ and columns $1, \ldots, j$.
We define the matrix
\begin{align*}
    D_n & : = \begin{pmatrix}
        I_n & \zeros\\
        \ones^\intercal & 1 \\
        \ones^\intercal & 1
    \end{pmatrix} \in \R^{(n+2) \times (n+1)},
\end{align*}
where $I_n$ is the $n$-dimensional identity matrix, $\zeros$ and $\ones$ the all-zero and all-one vectors.
Let $\Delta_n= \{\xx\in\R^n_{\geq 0}\mid \|\xx\|_1 =1\}$ denote the probability simplex.
Throughout, let $\cC\subseteq \R^n$ be a nonempty compact convex set and let $\aff(\cC)$, $\dim(\cC)$, and $\cV = \vertices(\cC)$ denote the affine hull, dimension, and set of vertices of $\cC$, respectively. If $\cC$ is a polytope, let $\faces(\cC)$ denote the sets of faces of $\cC$. 
For a set $F\subseteq \R^n$, $\conv(F)$ and $\relint(F)$ are the convex hull and relative interior of $F$, respectively. For $F, G \subseteq \R^n$ and $\xx\in \R^n$, $\dist{}(F, G) = \inf_{\yy\in F, \zz\in G} \|\yy - \zz\|_2$ is the Euclidean distance between $F$ and $G$ and $\dist{}(x, F) = \inf_{\yy\in F} \|\yy - \xx\|_2$ is the Euclidean point-set distance between $\xx$ and $F$.
A continuously differentiable function $f\colon \cC \to \R$ is \emph{$L$-smooth} over $\cC$ with $L > 0$ if
$
    f(\yy) \leq f(\xx) + \langle \nabla f(\xx), \yy - \xx\rangle + \frac{L}{2}\|\yy-\xx\|_2^2 \; \forall \xx,\yy\in \cC.
$
The function is \emph{$\mu$-strongly convex} with $\mu > 0$ if
$
f(\yy) \geq f(\xx) + \langle \nabla f(\xx), \yy - \xx\rangle + \frac{\mu}{2}\|\yy-\xx\|_2^2 \;\forall \xx,\yy\in \cC.
$

%
The pyramidal width \citep{lacoste2015global} is equivalent to the definition below by \citet{pena2019polytope}.
\begin{definition}[Pyramidal width]
    Let $\emptyset \neq \cC\subseteq \R^n$ be a polytope with vertex set $\cV$. The \emph{pyramidal width} of $\cC$ is 
        $
            \omega := \min_{F\in \faces{}(\cC), \emptyset \subsetneq F \subsetneq \cC} \dist{}(F, \conv(\cV \setminus F)).
        $
\end{definition}
\section{Amenable algorithms}\label{sec:steps_towards_pivoting}

We will begin by introducing two key concepts: a) \emph{Carathéodory-amenable algorithms} (\caa), outlined in Algorithm~\ref{alg:caa}, which are well-suited for use with the pivoting meta-algorithm (\pm); and b) \emph{convex-combination-agnostic} properties, which are inherent characteristics of \caa{}s that are preserved by the \pm.

In particular, we demonstrate that \fw{} algorithms such as vanilla \fw{}, Algorithm~\ref{alg:fw}, the away-step Frank-Wolfe algorithm (\afw{}), Algorithm~\ref{alg:afw}, and the blended pairwise Frank-Wolfe algorithm (\bpfw{}), Algorithm~\ref{alg:bpfw}, are \caa{}s. Furthermore, we prove that most convergence rates for \fw, \afw, and \bpfw are convex-combination-agnostic. Thus, \fw, \afw, or \bpfw modified with \pm regularly enjoy the same convergence rates as the original algorithms. 

\subsection{Carathéodory-amenable algorithms}
\begin{algorithm}[t]
\SetKwInput{Input}{Input} \SetKwInput{Output}{Output}
\SetKwComment{Comment}{$\triangleright$\ }{}
\caption{Carathéodory-amenable algorithm (\caa) [Template]}\label{alg:caa}
\Input{$\xx^{(0)}\in\cV$.
} 
  \Output{$\aalpha^{(T)}\in\Delta_{|\cV|}$, $\cS^{(T)} = \{\ss\in \cV\mid \alpha_\ss^{(T)} > 0\}$, and $\xx^{(T)}\in\cC$, such that $\xx^{(T)} = \sum_{\ss\in \cS^{(T)}} \alpha_\ss^{(T)} \ss$.}
  \hrulealg
  $\aalpha^{(0)}\in\Delta_{|\cV|}$ s.t. for all $\ss\in\cV$,  $\alpha_\ss^{(0)} \gets\begin{cases}
            1,& \text{if} \ \ss = \xx^{(0)}\\
            0, & \text{if} \ \ss\in\cV\setminus\{\xx^{(0)}\}
            \end{cases}$\\
            $\cS^{(0)} \gets \{\ss \in \cV \mid \alpha_\ss^{(0)} > 0\} = \{\xx^{(0)}\}$\\
  \For{$t = 0, 1, \ldots, T - 1$}{
    $(\aalpha^{(t+1)}, \cS^{(t+1)}, \xx^{(t+1)}) \gets \ccu(\aalpha^{(t)}, \cS^{(t)},\xx^{(t)})$ \Comment*[f]{see Algorithm~\ref{alg:ccu}}\label{line:caa_cau_call}
  }
\end{algorithm}

%
Consider an algorithm that represents each iterate $\xx^{(t)}$ as a convex combination, i.e., $\xx^{(t)} = \sum_{\ss \in \cS^{(t)}} \alpha_\ss^{(t)}\ss$. Here, $\aalpha^{(t)}\in \Delta_{|\cV|}$ and $\cS^{(t)} = \{\ss\in\cV \mid \alpha_{\ss}^{(t)} > 0 \}$ denote the \emph{weight vector} and \emph{active set} at iteration $t\in\{0, 1, \ldots, T\}$, respectively.
Most \fw variants then perform one of the following types of updates:

\paragraph{\fw update:} The algorithm shifts weight from all vertices in the active set to a single vertex, that is, $\xx^{(t+1)} = \xx^{(t)} + \eta^{(t)}(\vv^{(t)} - \xx^{(t)})$, where $\vv^{(t)}\in\cV$ and $\eta^{(t)} \in [0, 1]$. See, e.g., Line~\ref{line:fw_step} in \fw{} (Algorithm~\ref{alg:fw}).

\paragraph{Away update:} The algorithm shifts weight from a single vertex in the active set to all other vertices in the active set. That is, $\xx^{(t+1)} = \xx^{(t)} + \eta^{(t)}(\xx^{(t)} - \aa^{(t)})$, where $\aa^{(t)}\in \cS^{(t)}$ and $\eta^{(t)}\in [0, \alpha^{(t)}_{\aa^{(t)}}/(1- \alpha^{(t)}_{\aa^{(t)}})]$. See, e.g., Line~\ref{line:afw_away_step} in \afw{} (Algorithm~\ref{alg:afw}).

\paragraph{Pairwise update:} Assuming that $\eta^{(t)}\neq 1$, the algorithm performs both an away and a \fw update simultaneously, that is, $\xx^{(t + 1)} = \xx^{(t)} + \eta^{(t)} (\vv^{(t)} -\aa^{(t)})$, where $\vv^{(t)} \in \cV$, $\aa^{(t)} \in \cS^{(t)}$, $\vv^{(t)} \neq \aa^{(t)}$, and $\eta^{(t)}\in [0, \alpha^{(t)}_{\aa^{(t)}}]$. See, e.g., Line~\ref{line:bpfw_local_pairwise_step} in \bpfw{} (Algorithm~\ref{alg:bpfw}).
To see that a pairwise update is equivalent to an away update followed by an \fw update, consider the auxiliary vector
$
    \yy^{(t)} = \xx^{(t)} + \frac{\eta^{(t)}}{1-\eta^{(t)}} (\xx^{(t)} - \aa^{(t)})
$
obtained after performing an away update with step-size $\frac{\eta^{(t)}}{1-\eta^{(t)}}$.
Then, as required:
\begin{align*}
\xx^{(t + 1)} & = \xx^{(t)} + \eta^{(t)}(\vv^{(t)} - \aa^{(t)})\\
    & = \xx^{(t)} (1 - \eta^{(t)} + \frac{\eta^{(t)}}{1-\eta^{(t)}} (1-\eta^{(t)})) \nonumber\\
    & + \eta^{(t)}\vv^{(t)} - \frac{\eta^{(t)}}{1-\eta^{(t)}}(1-\eta^{(t)})\aa^{(t)} \nonumber\\
    & = \yy^{(t)} + \eta^{(t)}(\vv^{(t)} - \yy^{(t)}).\nonumber
\end{align*}

\begin{algorithm}[t]
\SetKwInput{Input}{Input} \SetKwInput{Output}{Output}
\SetKwComment{Comment}{$\triangleright$\ }{}
\caption{Convex-combination update (\ccu), $(\bbeta, \cT, \yy) = \ccu(\aalpha, \cS, \xx)$}\label{alg:ccu}
\Input{$\aalpha\in\Delta_{|\cV|}$, $\cS = \{\ss\in \cV\mid \alpha_\ss > 0\}$, and $\xx\in\cS$, such that $\xx = \sum_{\ss\in \cS} \alpha_\ss \ss$.} 
  \Output{$\bbeta\in\Delta_{|\cV|}$, $\cT = \{\ss\in \cV\mid \beta_\ss > 0\}$, and $\yy\in\cC$, such that $\yy = \sum_{\ss\in \cT} \beta_\ss \ss$ and either $\cT\subseteq \cS$ or there exists exactly one $\vv\in\cT\setminus\cS$ such that $\cT\subseteq \cS\cup\{\vv\}$.
}
\end{algorithm}
These updates are captured by the more general \emph{convex-combination updates} (\ccu{}s) formalized in Algorithm~\ref{alg:ccu}. 
We refer to algorithms that perform repeated \ccu{}s as Carathéodory-amenable algorithms (\caa{}s), see Algorithm~\ref{alg:caa}. Several comments are warranted. First, the running examples of this paper, \fw{}, \afw{}, and \bpfw{}, are \caa{}s. Second, despite the focus of this work on \fw algorithms, the class of \caa{}s is general enough to potentially capture other methods, such as simplicial decompositions \citep{bettiol2024oracle} or the nearest-extreme point variant \citep{garber2021frank}. We leave the applicability of our framework for algorithms forming iterates as linear or conic combinations (such as matching pursuit) to future research.
%
Finally, we highlight that a \ccu{} does not require $\vv$ to be obtained via an (exact) \lmo. Thus, the lazified variants of \fw, \afw, and \bpfw  are also \caa{}s.

\subsection{Convex-combination-agnostic properties}

We now formalize in the following definition the concept of properties of \caa{}s that remain unchanged when one convex representation is replaced by another. As demonstrated in Section~\ref{sec:pm}, the pivoting meta-algorithm (\pm) ensures the preservation of these properties.

\begin{definition}[Convex-combination-agnostic property]\label{def:cca}
   Consider a \caa and suppose that the output of the algorithm provably satisfies a property $\cP$. We say that $\cP$ is \emph{convex-combination-agnostic} if property $\cP$ also holds if \caa is modified by replacing Line~\ref{line:caa_cau_call} in Algorithm~\ref{alg:caa} with the following two lines:
   \begin{align*}
    & {\footnotesize\textbf{4a)}} \  (\bbeta^{(t+1)}, \cT^{(t+1)}, \xx^{(t+1)}) \gets \ccu(\aalpha^{(t)}, \cS^{(t)},\xx^{(t)}) \\
     & \begin{aligned}
        {\footnotesize\textbf{4b)}} \ (\aalpha^{(t+1)}, \cS^{(t+1)}) \gets & \ \text{constructed via any procedure}\\
        & \text{that guarantees that}\\
        & \xx^{(t + 1)} = \sum_{\ss \in \cS^{(t + 1)}} \alpha_\ss^{(t + 1)}\ss,\\
        & \text{where} \ \aalpha^{(t+1)}\in\Delta_{|\cV|} \ \text{and} \\
        & \cS^{(t+1)} = \{\ss\in \cT^{(t)}\mid \alpha_\ss^{(t+1)} > 0\}.
    \end{aligned}
\end{align*}
\end{definition}
Most properties are convex-combination-agnostic, including convergence rates of the \fw variants that we consider as running examples here and which we will represent in the remainder of this section.

\begin{theorem}[Sublinear convergence rate of \fw{}]\label{thm:fw_sublinear}
    Let $\cC\subseteq \R^n$ be a compact convex set of diameter $\delta > 0$ and let $f\colon \cC \to \R$ be a convex and $L$-smooth function. Then, for the iterates of Algorithm~\ref{alg:fw} (\fw) with line-search, short-step, or open-loop step-size rule $\eta^{(t)} = \frac{2}{t+2}$, the convergence guarantee
    $
    f(\xx^{(t)}) - \min_{\xx\in\cC}f(\xx) \leq \frac{2 L \delta^2}{t+2}
    $
    is convex-combination-agnostic.
\end{theorem}

\citet{tsuji2022pairwise} showed convergence rate guarantees for \bpfw{} similar to those of of \afw{}.
Below, we present both the general sublinear rate as well as the linear rate for the case of $\cC$ being a polytope.

\begin{theorem}[Sublinear convergence rates of \afw{} and \bpfw{}]\label{thm:afw_bpfw_sublinear}
    Let $\cC\subseteq \R^n$ be a compact convex set of diameter $\delta > 0$ and let $f\colon \cC \to \R$ be a convex and $L$-smooth function. Then, for the iterates of Algorithms~\ref{alg:afw} (\afw{}) and~\ref{alg:bpfw} (\bpfw{}) with line-search, the convergence guarantee
    $
        f(\xx^{(t)}) - \min_{\xx\in\cC}f(\xx)\leq \frac{4 L \delta^2}{t}
    $
    is convex-combination-agnostic.
\end{theorem}

\begin{theorem}[Linear convergence rates of \afw{} and \bpfw{}]\label{thm:afw_bpfw_linear}
    Let $\cC\subseteq \R^n$ be a polytope of diameter $\delta > 0$ and pyramidal width $\omega > 0$, and let $f\colon \cC \to \R$ be a $\mu$-strongly convex and $L$-smooth function. Then, for the iterates of Algorithms~\ref{alg:afw} (\afw{}) and~\ref{alg:bpfw} (\bpfw{}) with line-search, the convergence guarantee
    $f(\xx^{(t)}) - \min_{\xx\in\cC}f(\xx) \leq (f(\xx^{(0)}) - \min_{\xx\in\cC}f(\xx))\exp (-\frac{t}{2} \min \{\frac{1}{2}, \frac{\mu \omega^2}{4L\delta^2}\})$ is convex-combination-agnostic.
\end{theorem}

\section{The pivoting meta algorithm}\label{sec:pm}
\begin{algorithm}[t]
\SetKwInput{Input}{Input} \SetKwInput{Output}{Output}
\SetKwComment{Comment}{$\triangleright$\ }{}
\caption{Pivoting meta algorithm (\pm), $(\aalpha^{(T)}, \cS^{(T)}, \xx^{(T)})= \pm(\xx^{(0)})$}\label{alg:pm}
\Input{$\xx^{(0)} \in \cV$.} 
  \Output{$\aalpha^{(T)}\in\Delta_{|\cV|}$, $\cS^{(T)} = \{\ss\in \cV\mid \alpha_\ss^{(T)} > 0\}$, and $\xx^{(T)}\in\cC$, such that $\xx^{(T)} = \sum_{\ss\in \cS^{(T)}} \alpha_\ss^{(T)} \ss$.}
  \hrulealg
     $M^{(0)} \gets \left(\tilde{\xx}^{(0)}, D_n\right)\in \R^{(n + 2) \times (n + 2)}$\label{line:pm_M0}\\
      $\aalpha^{(0)}\in\Delta_{|\cV|}$ s.t. for all $\ss\in\cV$,  $\alpha_\ss^{(0)} \gets\begin{cases}
            1,& \text{if} \ \ss = \xx^{(0)}\\
            0, & \text{if} \ \ss\in\cV\setminus\{\xx^{(0)}\}
            \end{cases}$\label{line:pm_a0}\\
    $\cS^{(0)} \gets \{\ss \in \cV \mid \alpha_\ss^{(0)} > 0\} = \{\xx^{(0)}\}$\label{line:pm_S0}\\
    \For(\label{line:pm_for_loop}){$t= 0, 1, \ldots, T-1 $}{
            $(\bbeta^{(t+1)}, \cT^{(t+1)}, \xx^{(t+1)}) \gets \ccu(\aalpha^{(t)}, \cS^{(t)}, \xx^{(t)})$\Comment*[f]{see Algorithm~\ref{alg:ccu}} \label{line:pm_ccu} \\
            $(M^{(t+1)}, \aalpha^{(t+1)}, \cS^{(t+1)}) \gets \asc(M^{(t)}, \beta^{(t+1)}, \cT^{(t+1)}, \cT^{(t+1)}\setminus\cS^{(t)})$\Comment*[f]{see Algorithm~\ref{alg:asc}}\label{line:pm_asc}
            }
\end{algorithm}

We now introduce the \emph{pivoting meta algorithm (\pm)} in Algorithm~\ref{alg:pm}, a drop-in modification applicable to existing \caa{}s, as shown in Algorithm~\ref{alg:caa}. This ensures that the size of the modified active set $\cS^{(t)}$ remains bounded by $n+1$ for all $t \in \{0, 1, \ldots, T\}$, while preserving the convex-combination-agnostic properties of the original \caa{}.

The main idea behind \pm{} is to modify the active set obtained from \ccu{} in Line~\ref{line:pm_ccu} with the \emph{active set cleanup algorithm} (\asc{}), presented in Algorithm~\ref{alg:asc}. 
\asc{} takes as arguments the new weight vector $\bbeta^{(t+1)}\in\Delta_{|\cV|}$, the new active set $\cT^{(t+1)} = \{\ss\in\cV\mid \beta^{(t+1)}_\ss > 0\}$, the set difference between the upcoming and the current active set $\cD^{(t)}:=\cT^{(t+1)}\setminus\cS^{(t)}\subseteq \cT^{(t+1)}$, and a matrix $M^{(t)}\in\R^{(n+2)\times(n+2)}$, such that the following hold:
\begin{enumerate}
    \item $M^{(t)}$ is invertible, $M^{(t)}_{n+1,:} \geq \zeros^\intercal$, and $M^{(t)}_{n+2,:} \geq \ones^\intercal$.
    \item  For all $i\in[n+2]$, $M^{(t)}_{n+1,i} = 0$ implies that there exists an $\ss\in\cS^{(t)}$ such that $\tilde{\ss} = M^{(t)}_{:, i}$.
    \item For all $\ss\in\cS^{(t)}$ there exists an $i\in[n+2]$ such that $M^{(t)}_{:, i} = \tilde{\ss}$,
\end{enumerate}
In Line~\ref{line:pm_asc} of \pm{}, \asc{} then constructs the matrix $M^{(t+1)}$, the modified weight vector $\aalpha^{(t+1)}\in\Delta_{|\cV|}$ and the modified active set $\cS^{(t+1)} = \{\ss\in\cV\mid \alpha^{(t+1)}_\ss > 0\}$, such that
$\xx^{(t+1)} = \sum_{\ss\in\cT^{(t+1)}}\beta_\ss^{(t+1)}\ss = \sum_{\ss\in\cS^{(t+1)}}\alpha_\ss^{(t+1)}\ss$, $|\cS^{(t+1)}| \leq n + 1$, and the three properties above are now satisfied for $t+1$.

By the definition of \ccu{}s, $\cD^{(t)}$ is either empty or contains exactly one vertex $\vv^{(t)}\in\cT^{(t+1)}\setminus\cS^{(t)}$.
The former case is straightforward as the size of the active set does not increase. In the latter case, \asc{} performs a pivoting update akin to the simplex algorithm that guarantees that $|\cS^{(t+1)}|\leq n + 1$. Performing pivot updates necessitates maintaining the matrix $M^{(t)}$ throughout \pm{}'s execution.
We formalize the properties of \asc{} below.
\begin{algorithm}[ht!]
\SetKwInput{Input}{Input} \SetKwInput{Output}{Output}
\SetKwComment{Comment}{$\triangleright$\ }{}
\caption{Active set cleanup algorithm (\asc), $(M, \aalpha, \cS) = \asc(N, \bbeta, \cT, \cD)$}\label{alg:asc}
\Input{$N\in \R^{(n+2) \times (n+2)}$ invertible, $\bbeta\in\Delta_{|\cV|}$, $\cT = \{\ss\in \cV\mid \beta_\ss > 0\}$, and $\cD\subseteq\cT$ such that $|\cD| \leq 1$.} 
  \Output{$M\in \R^{(n+2) \times (n+2)}$ invertible, $\aalpha\in\Delta_{|\cV|}$, and $\cS = \{\ss\in \cV\mid \alpha_\ss > 0\}$.}
  \hrulealg
            \uIf(\label{line:asc_if}){$\cD = \emptyset$}
            {
             $\llambda\in\Delta_{n+2}$ s.t. for all $i\in[n+2]$,  $\lambda_i \gets\begin{cases}
            \beta_\ss, & \text{if} \ N_{:,i} = \tilde{\ss} \ \text{for some} \ \ss\in\cT\\
            0,& \text{else}
            \end{cases}$\label{line:asc_if_lambda}\\
            $Q \gets N$\label{line:asc_if_Q}
            }
            \Else(\label{line:asc_else}){
            $\vv\in\cD$ \Comment*[f]{$\vv$ is unique}\label{line:asc_else_v}\\
            $A \gets (N, \tilde{\vv}) \in\R^{(n+2) \times (n+3)}$ \label{line:asc_A}\\
            $\mmu\in\Delta_{n+3}$ s.t. for all $i\in[n+3]$,  $\mu_i \gets\begin{cases}
            \beta_\ss, & \text{if} \ A_{:,i} = \tilde{\ss} \ \text{for some} \ \ss\in\cT\\
            0,& \text{else}
            \end{cases}$\label{line:asc_else_mu}\\
            {$\rr \gets -N^{-1} \tilde{\vv} \in \R^{n+2}$}\\ 
            {$k\in \argmin_{i\in[n+2],  r_i < 0} -\mu_i / r_i$}\label{line:asc_else_k}\\
            {$\theta^* \gets -\mu_k / r_k \geq 0$}\label{line:asc_else_theta}\\
            $\llambda\in\Delta_{n+2}$ s.t. for all $i\in[n+2]$, $\lambda_i \gets\begin{cases}
            \theta^*, & \text{if} \  i = k\\
            \mu_i + \theta^* r_i,& \text{if} \  i \neq k\\
            \end{cases}$\label{line:asc_else_lambda}\\
            $Q\in\R^{(n+2)\times(n+2)}$ s.t. for all $i\in[n+2]$, $Q_{:,i}\gets \begin{cases}
            \tilde{\vv},& \text{if} \  i =  k\\
            N_{:, i},& \text{if} \  i \neq k
            \end{cases}$\label{line:asc_else_Q}\\
    }
    $\ell \in \{ i\in[n+2]  \mid Q_{n+1,i} \neq 0\}$\\
    $M\in\R^{(n+2)\times(n+2)}$ s.t. for all $i\in[n+2]$, $M_{:, i}\gets\begin{cases}
    Q_{:, i} + Q_{:, \ell},&\text{if} \ \lambda_i = 0 \ \& \ Q_{n+1,i} = 0\\
    Q_{:, i}, & \text{else}
    \end{cases}$\label{line:asc_M}\\
    $\aalpha\in\Delta_{|\cV|}$ s.t. for all $\ss\in\cV$, $\alpha_\ss \gets\begin{cases}
        \lambda_i, & \text{if} \ \exists \ i\in[n+2] \ \text{s.t.} \ \tilde{\ss} = M_{:, i}\\
        0, & \text{for all other} \ \ss\in\cV
        \end{cases}$\label{line:asc_alpha}\\
        $\cS \gets \{\ss \in \cV \mid \alpha_\ss > 0\}$\label{line:asc_S}
\end{algorithm}

\begin{proposition}[Properties of \asc{}]\label{prop:asc}
Let $\cC\subseteq\R^n$ be a compact convex set, let $\cV = \vertices(\cC)$, let $N\in \R^{(n+2) \times (n+2)}$, let $\bbeta\in\Delta_{|\cV|}$, let $\cT = \{\ss\in \cV\mid \beta_\ss > 0\}$, and let $\cD \subseteq \cT$ such that $|\cD|\leq 1$.
Assume that the following hold:
\begin{enumerate}
    \item \label{assumpiton:asc_N_invertible} $N$ is invertible, $N_{n+1,:} \geq \zeros^\intercal$, and $N_{n+2,:} \geq \ones^\intercal$.
    \item \label{assumption:asc_N_tO^{(t)}} For all $i\in[n+2]$, $N_{n+1,i} = 0$ implies that there exists an $\ss\in\cT\setminus\cD$ such that $\tilde{\ss} = N_{:, i}$.
    \item \label{assumption:asc_T_to_N} For all $\ss\in\cT\setminus\cD$, there exists an $i\in[n+2]$ such that $N_{:, i} = \tilde{\ss}$.
\end{enumerate}
Let $(M, \aalpha, \cS) = \asc(N, \bbeta, \cT, \cD)$, where $M\in\R^{(n+2) \times (n+2)}$, $\aalpha\in\Delta_{|\cV|}$ and $\cS = \{\ss \in \cV \mid \alpha_\ss > 0\}$.
Then we have:
\begin{enumerate}[resume]
    \item \label{property:asc_M_invertible} $M$ is invertible, $M_{n+1,:} \geq \zeros^\intercal$, and $M_{n+2,:} \geq \ones^\intercal$.
    \item \label{property:asc_lambda_i} For all $i\in[n+2]$, $\lambda_i > 0$ implies that there exists an $\ss\in\cS\cap\cT$ such that $\tilde{\ss} = M_{:, i} = Q_{:,i}$.
    \item \label{property:asc_equality_S_T} $\xx:=\sum_{\ss\in \cT} \beta_\ss \ss = \sum_{\ss\in \cS} \alpha_\ss \ss$.
    \item \label{property:asc_M_to_S}
    For all $i\in[n+2]$, $M_{n+1,i} = 0$ implies that there exists an $\ss\in\cS$ such that $\tilde{\ss} = M_{:, i}$.
    \item \label{property:asc_S_to_M}
    For all $\ss\in\cS$, there exists an $i\in[n+2]$ such that $M_{:, i} = \tilde{\ss}$.
    \item  \label{property:asc_S_subseteq_T} It holds that $\cS \subseteq \cT$ and $|\cS|\leq n + 1$.
\end{enumerate}
\end{proposition}

As a direct consequence, we formalize the properties of \pm{} in the main result of the paper below. 
\begin{theorem}[Properties of \pm]\label{thm:pm}
   Let $\cC\subseteq \R^n$ be a compact convex set and $\xx^{(0)} \in \vertices(\cC) = \cV$. Given a specific \caa, let $(\aalpha^{(T)}, \cS^{(T)}, \xx^{(T)})= \pm(\xx^{(0)})$ denote the output of its modification with \pm. Then, for all $t \in \{0,1,\ldots,T\}$ the following hold:
    \begin{enumerate}
        \item \label{property:pm_convex_combination} $\aalpha^{(t)}\in \Delta_{|\cV|}$ and $\cS^{(t)} = \{\ss\in \cV\mid \alpha_\ss^{(t)} > 0 \}$.
        \item\label{property:pm_M_invertible} $M^{(t)}\in\R^{(n+2)\times(n+2)}$ is invertible, $M_{n+1,:}^{(t)} \geq \zeros^\intercal$, and $M_{n+2,:}^{(t)} \geq \ones^\intercal$.
        \item \label{property:pm_x_equal_to_alpha_S_equal_to_beta_T} $\xx^{(t)}  = \sum_{\ss\in \cT^{(t)}} \beta_\ss^{(t)} \ss = \sum_{\ss\in \cS^{(t)}} \alpha_\ss^{(t)} \ss$. For $t=0$, let $\cT^{(0)} := \cS^{(0)}$ and $\bbeta^{(0)} : = \aalpha^{(0)}$.
        \item \label{property:pm_active_set_size_reduction} $\cS^{(t)}\subseteq\cT^{(t)}$ and $|\cS^{(t)}| \leq  n + 1$.
    \end{enumerate}
    Moreover, \pm's output satisfies the same convex-combination-agnostic properties \caa's output would satisfy.
\end{theorem}

We will later see that the dimension dependence in Property~\ref{property:pm_active_set_size_reduction} of Theorem~\ref{thm:pm} can be replaced with $\dim(\cC)$, see Corollary~\ref{cor:active_set_pm}.

\subsection{Comparison to IRR}\label{sec:irr}
A first proposal for the reduction of the active set cardinality in \caa algorithms motivated by Carathéodory's theorem was presented in \citet{beck2017linearly} as the \emph{incremental representation reduction algorithm} (\irr{}).
It maintains two matrices throughout the algorithm, $T$ and $W$, with $T$ a product of elementary matrices and $W$ a matrix in row echelon form.
The latter will, in general, not be sparse, and the whole analysis of the authors is not considering the potential sparsity of vertices, and thus of the resulting matrices.
In contrast, \pm operates directly on a matrix formed by the (extended) vertices and explicitly includes the linear system solving step for which the algorithmic details remain at the discretion of the implementation, instead of relying on the formation of a row-echelon matrix which is rarely the preferred choice to solve sparse linear systems.
The workspace required by \pm only consists of $M^{(t)}$ and at most one additional column, which is of a fixed size of only $\cO(n_v (n+3))$, with $n_v$ the number of structural non-zero terms in a single vertex, assuming Line~\ref{line:asc_M} does not merge vertex columns into non-vertex ones in too many iterations.
This bound is typically much lower than \irr{}'s space requirements of order $\cO(n^2)$ for many applications in which the support of vertices is small.

Finally, as detailed in Section~\ref{sec:active_set_identification}, numerical errors in rank-one updates quickly accumulate beyond a practical level to compute the weights of the active set.
\irr{} requires maintaining the row-echelon matrix throughout iterations and does not specify a mechanism to start from a given non-singleton active set, which means numerical errors inevitably accumulate throughout the algorithm. \pm on the other hand can leverage rank-one updates and at any step recompute a factorization from scratch to solve the sparse linear system $M \rr = \vv$,
or leverage any alternative algorithm for linear systems, making it more flexible for numerically challenging instances.
Furthermore, \pm requires significant computational work only when $\cT^{(t+1)}$ contains a vertex not already contained in $\cS^{(t)}$, that is, when a new vertex is introduced into the active set. On steps operating on known vertices only, see Line~\ref{line:asc_if} of \asc{} (Algorithm~\ref{alg:asc}), the only subroutine in $\cO(n_v n_d)$ is Line~\ref{line:asc_M}, with $n_d$ the number of vertices dropped at the given step. In practice, $n_d$ remains quite small as one rarely drops a lot of vertices from the active set. In contrast, \irr{} always incurs a computational overhead of $\cO(n^2)$.

\section{Active set identification results}\label{sec:active_set_identification}
We will now present improvements to the bound $|\cS^{(t)}|\leq n + 1$ established for \pm in Theorem~\ref{thm:pm}. 
First, we refine the bound to $|\cS^{(t)}|\leq \dim(\cC) + 1$ for any iteration $t$. 
Then, we establish that if there exists an iteration $\Rfwt$ such that $\cS^{(t)}\subseteq\cC^* \subseteq\cC$ for all $t \geq \Rfwt$, then $|\cS^{(t)}|\leq \dim(\cC^*) + 1$ instead of $|\cS^{(t)}|\leq n + 1$ for all $t \geq \Rfwt$. This result is tied to so-called active set identification properties of some FW variants, known for a long time in specific settings \citep{guelat1986some}, and recently generalized in \citet{bomze2020active} for \afw{} when $\cC$ is a polytope and some other mild assumptions are met.\footnote{The result of \citet{bomze2020active} is proved in their paper for \afw{}. In Appendix~\ref{appendix:bpfw_asi}, we prove that a similar result also holds for \bpfw{}.}

Below, we present the main result of this section.
\begin{theorem}[Active set identification with \pm]\label{thm:active_set_pm}
    Let $\cC\subseteq \R^n$ be a compact convex set and $\xx^{(0)} \in \vertices(\cC) = \cV$. Given a specific Algorithm~\ref{alg:caa} (\caa), let $(\aalpha^{(T)}, \cS^{(T)}, \xx^{(T)})= \pm(\xx^{(0)})$ denote the output of its modification with Algorithm~\ref{alg:pm} (\pm). Suppose that there exists an iteration $\Rfwt\in\{0,1,\ldots, T\}$ such that $\cS^{(t)}\subseteq\cC^*$ for all $t\in\{\Rfwt, \Rfwt + 1 \ldots, T\}$, where $\emptyset\neq\cC^*\subseteq\cC$. Then, 
    $|\cS^{(t)}| \leq  \dim(\cC^*) + 1$ for all $t\in\{\Rfwt, \Rfwt + 1 \ldots, T\}$.
\end{theorem}

The result above implies that the bound $|\cS^{(t)}| \leq n + 1$ in Theorem~\ref{thm:pm} can be refined to $|\cS^{(t)}| \leq \dim(\cC) + 1$ for any iteration $t$.
\begin{corollary}\label{cor:active_set_pm}
   Let $\cC\subseteq \R^n$ be a compact convex set and $\xx^{(0)} \in \vertices(\cC) = \cV$. Given a specific Algorithm~\ref{alg:caa} (\caa), let $(\aalpha^{(T)}, \cS^{(T)}, \xx^{(T)})= \pm(\xx^{(0)})$ denote the output of its modification with Algorithm~\ref{alg:pm} (\pm). Then, 
    $|\cS^{(t)}| \leq \dim(\cC) + 1$ for all $t\in\{0, 1, \ldots, T\}$.
\end{corollary}
We now focus on the application of Theorem~\ref{thm:active_set_pm} to the setting characterized in \citet{bomze2020active}.

Consider the optimization problem \eqref{eq:opt} with $\cC\subseteq\R^n$ a polytope, $f\colon \cC \to \R$ a convex and $L$-smooth function, and the set of optimal solutions $\argmin_{\xx\in\cC} f(\xx)$ lying in the relative interior of a face $\cC^* \subseteq \cC$.
This setting is particularly relevant to the \fw{} community, encompassing applications such as sparse signal recovery, sparse regression, and sparse logistic regression, where \fw{} variants construct iterates that are sparse convex combinations of vertices of the feasible region. Second, vanilla \fw{} with line-search or short-step is known to exhibit zig-zagging behavior and, in some cases, converges at a rate of at most $\Omega(1/t^{1+\epsilon})$ for any $\epsilon > 0$ \citep{wolfe1970convergence}. This motivated the development of accelerated variants \citep{lacoste2015global,garber2016linear,braun2019blended,combettes2020boosting,garber2021frank} and the exploration of alternative step sizes \citep{wirth2023acceleration,wirth2023accelerated} to surpass the lower bound established by \citet{wolfe1970convergence}. Finally, the identification of the optimal face $\cC^*$ is a crucial step known as active set identification; not to be confused with the active sets maintained by the \caa's. Once $\cC^*$ is identified, the optimization problem \eqref{eq:opt} simplifies to one over $\cC^*$, which often has a much lower dimension compared to $\cC$ \citep{bomze2019first, bomze2020active}.

Recently, \citet{bomze2020active} provided insights into the settings where the away-step Frank-Wolfe algorithm (\afw{}) guarantees the identification of the active set after a finite number of iterations. Here, we present a simplified version\footnote{\citet{bomze2020active} also provide sufficient conditions for their result to hold when there are multiple optimizers in the relative interior of an optimal face.} of their result, \citet[Theorem~C.1]{bomze2020active}, along with our observation regarding its convex-combination-agnostic nature.
\begin{theorem}[Theorem~C.1, \citealp{bomze2020active}]\label{thm:afw_active}
Let $\cC\subseteq\R^n$ be a polytope, let $f\colon\cC\to\R$ be a convex and $L$-smooth function, and suppose that $\xx^*\in\argmin_{\xx\in\cC}f(\xx)$ is unique and $\xx^*\in\relint(\cC^*)$, where $\cC^*\in\faces(\cC)$. Then, for $T$ large enough, for the iterations of Algorithm~\ref{alg:afw} (\afw) with line-search, there exists an iteration $\Rfwt \in \{0, 1, \ldots, T\}$ such that $\xx^{(t)}\in\cC^*$ for all $t\in\{\Rfwt, \Rfwt+1,\ldots, T\}$. This property is convex-combination-agnostic.
\end{theorem}

We similarly establish the active set identification property for \bpfw and its convex-combination agnosticity in Theorem~\ref{thm:bpfw_active}.
We obtain the following active set identification result for \pm applied to \afw and \bpfw.
\begin{corollary}[Active set identification with \pm{} applied to \afw{} and \bpfw{}]\label{corollary:asi_pm_afw}
Let $\cC\subseteq\R^n$ be a polytope, let $f\colon\cC\to\R$ be a convex and $L$-smooth function, and suppose that $\xx^*\in\argmin_{\xx\in\cC}f(\xx)$ is unique and $\xx^*\in\relint(\cC^*)$, where $\cC^*\in\faces(\cC)$.
Let $(\aalpha^{(t)}, \cS^{(t)}, \xx^{(t)})= \pm(\xx^{(0)})$ denote the iterations of Algorithms~\ref{alg:afw} (\afw) or~\ref{alg:bpfw} (\bpfw) with line-search modified with Algorithm~\ref{alg:pm} (\pm). Then, there exists an iteration $\Rfwt \geq 0$ such that $|\cS^{(t)}| \leq \dim(\cC^*) + 1$ for all $t\geq \Rfwt$.
\end{corollary}

\begin{figure*}
	\centering
	\begin{subfigure}[t]{1\textwidth}
		\centering
		\includegraphics[width=0.47\textwidth]{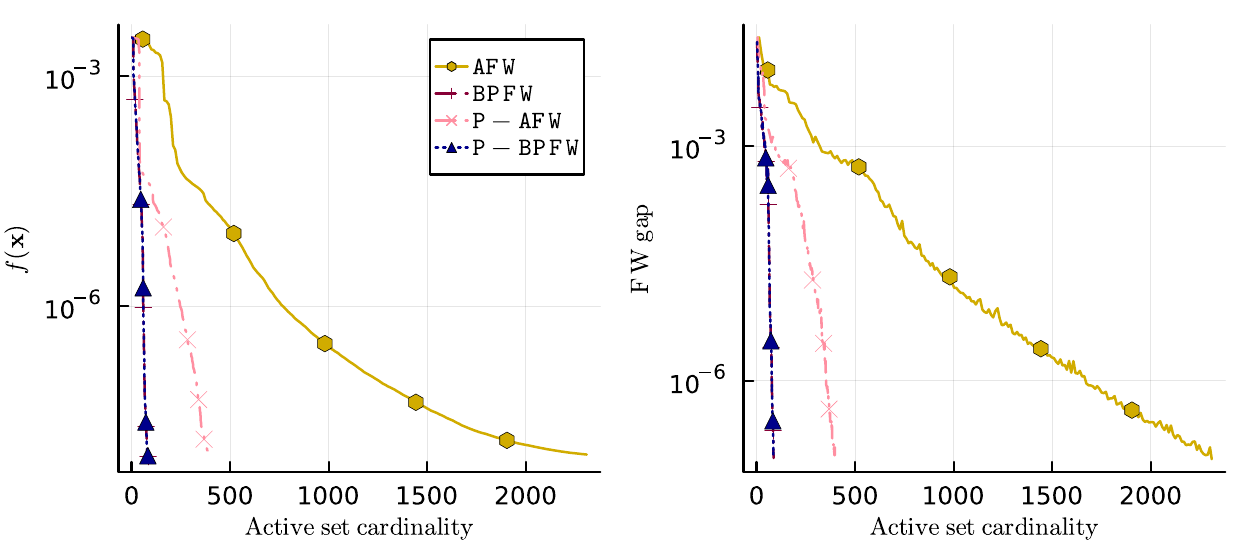}
		\includegraphics[width=0.47\textwidth]{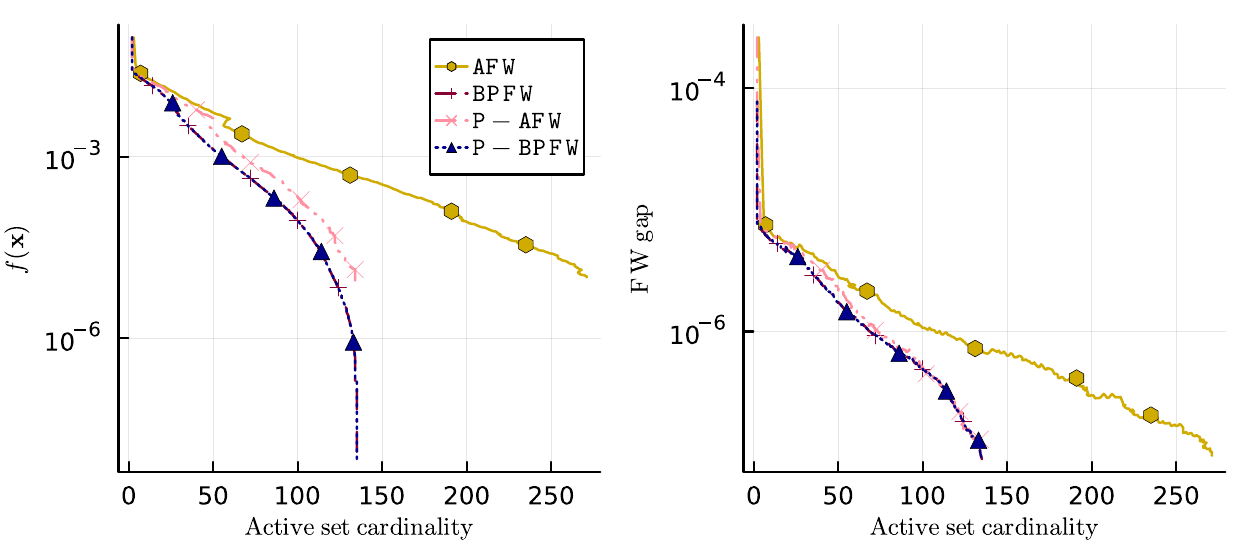}
	\end{subfigure}
	\hspace{0.05\textwidth}
	\begin{subfigure}[t]{1\textwidth}
		\centering
		\includegraphics[width=0.47\textwidth]{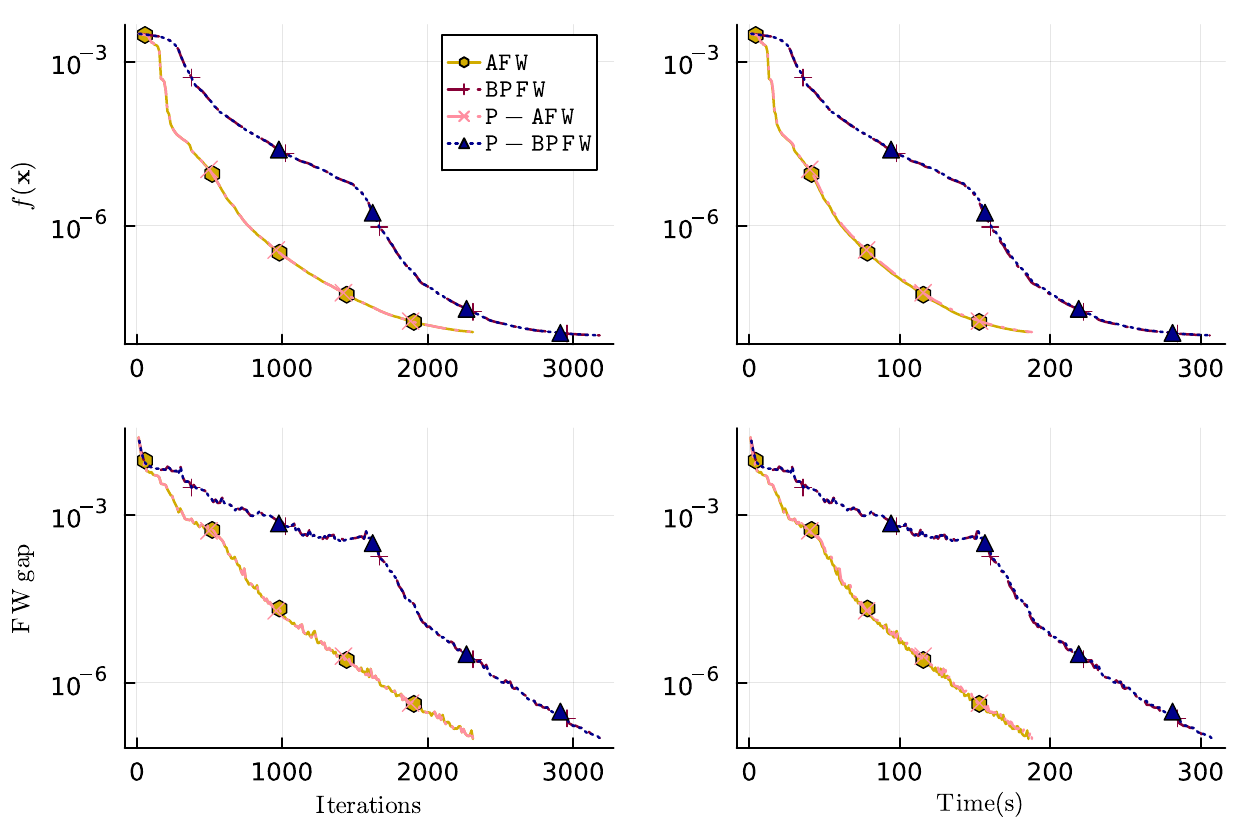}
		\includegraphics[width=0.47\textwidth]{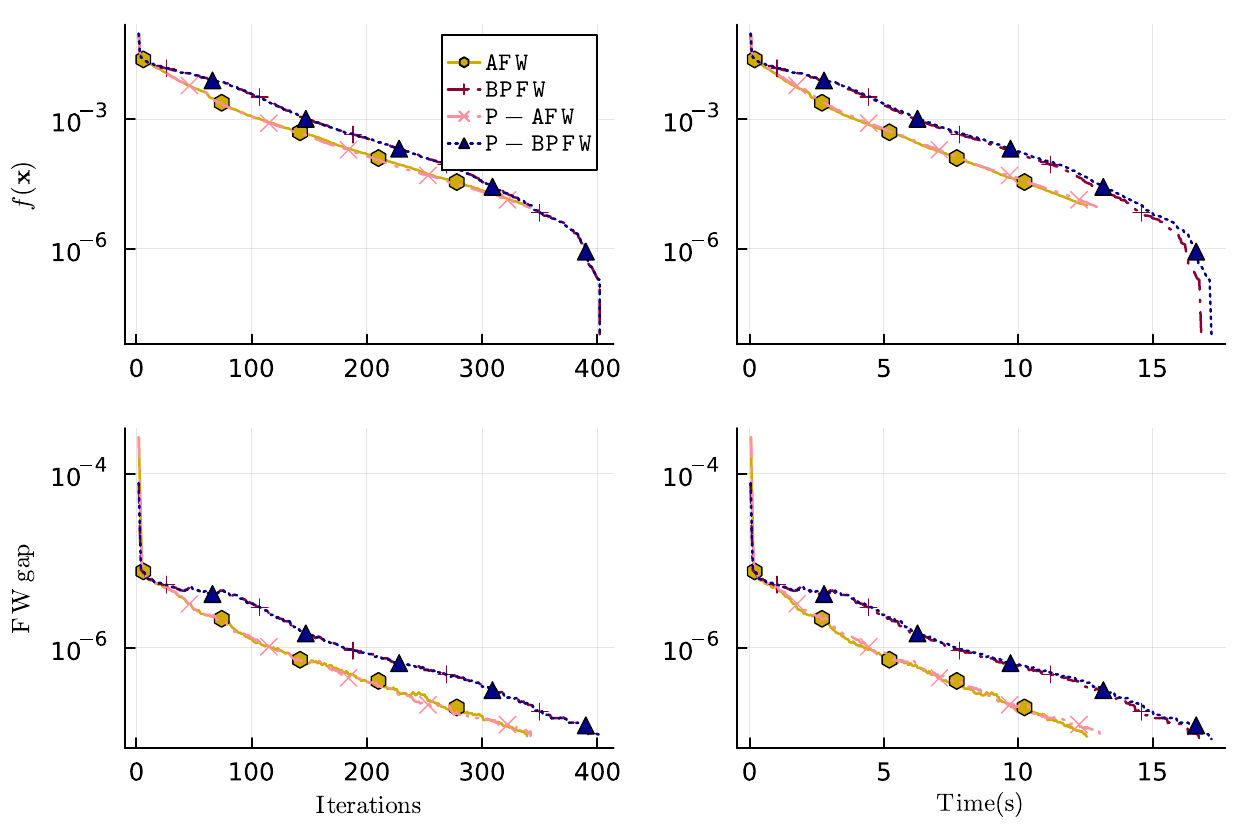}
	\end{subfigure}
	\caption{(Left) Logistic regression, $\tau=60$. (Right) Signal recovery, $\tau_f=20$, $m = 6000$, $n = 14000$. All variants are non-lazified. As can be seen in both cases, \pm can significantly improve sparsity while maintaining identical convergence rates and in cases where the algorithm (such as e.g., BPFW) produces already sparse iterates it does not harm the algorithm.}
	\label{fig:summary}
\end{figure*}

\section{Numerical experiments}\label{sec:numerical_experiments-body}

We will now provide a brief set of numerical experiments. We discuss numerical robustness and stability in Section~\ref{sec:algorithmic} and more extensive experiments as well as a detailed description of the setup in Section~\ref{sec:numerical_experiments} in the supplementary materials.

We assess \pm{} on efficiency in terms of function value
and FW gap convergence, and sparsity of the obtained solutions, compared to the standard and lazified versions (see \citep{braun2017lazifying} for details) of \afw{} and \bpfw{}, which also produce notably sparse solutions. Our algorithm is implemented in Julia \citep{bezanson2017julia} v1.9.2 and builds on the \texttt{FrankWolfe.jl} package \citep{besanccon2022frankwolfe}.
The sparse linear systems are solved with the LU decomposition of the \texttt{UMFPACK} library \citep{umfpack}.
Plots are log-linear, function values are shifted so that the minimum on each plot reaches $10^{-8}$. The prefix \texttt{L-} denotes the lazified version of an algorithm, the prefix \texttt{P-} for the \pm variant.

\subsection{Sparse logistic regression}

We run all algorithms on a logistic regression problem with an $\ell_1$-norm ball constraint:
\begin{align*}
	\min_{\xx : \|\xx\|_1 \leq \tau} \; & \frac1m \sum_{i=1}^{m} \log(1+\exp (-y_i \mathbf{a}_i^\intercal \xx).
\end{align*}
where $m$ is the number of samples, $\tau > 0$ is the $\ell_1$-norm ball radius, $y_i \in \{-1,1\}$ encodes the class and $\mathbf{a}_i$ the feature vector for the $i$th sample.
We use the Gisette dataset \citep{guyon2004result}, which contains 5000 features, we run logistic regression on the validation set containing only 1000 samples and thus more prone to overfitting without a sparsity-inducing regularization constraint.
The convergence of the pivoting variants of both \afw{} and \bpfw{} converge similarly in both function value and FW gap as their standard counterparts as shown in
Figure~\ref{fig:summary} (left).
Even though \bpfw{} typically maintains a smaller active set, it converges at a slower rate than the away-step \fw{} variants, both in function value and FW gap.
\texttt{P}-\afw{} drastically improves the sparsity of the \afw{} iterates, while maintaining the same convergence in function value and FW gap. This highlights one key property of our meta algorithm: it can be adapted to several \fw{} variants, benefiting from their convergence rate while improving their sparsity.

\subsection{Sparse signal recovery}

We assess \pm{} on a sparse signal recovery problem:
\begin{align*}
	\min_{\xx : \|\xx\|_1 \leq \tau}\, &  \|A\xx - \yy\|_2^2,
\end{align*}
with $A\in\mathbb{R}^{m\times n}$, $\yy\in\R^m$, $ n > m$.
We generate the entries of the sensing matrix $A$ i.i.d.~from a standard Gaussian distribution and $\yy$ by adding Gaussian noise with unit standard deviation to $A \xx_{\text{true}}$,
with $\xx_{\text{true}}$ an underlying sparse vector, with $30\%$ of non-zero terms, all taking entries sampled from a standard Gaussian distribution.
The radius $\tau$ is computed as $\tau = \|\xx_{\text{true}}\|_1 / \tau_f$ for different values of $\tau_f$.

Figure~\ref{fig:summary} (right) illustrates the results of the non-lazified version of \bpfw{} and \afw{} and their pivoting counterparts.
\texttt{P-\afw{}} converges at the same rate as \afw in function value and FW gap while being faster than both \bpfw{} variants, terminating before them, while maintaining an active set twice as sparse as \afw.

\clearpage

\subsubsection*{Acknowledgments}
We thank Zev Woodstock for providing valuable feedback for an earlier version of this manuscript.
Research reported in this paper was partially supported by the Deutsche Forschungsgemeinschaft (DFG, German Research Foundation) under Germany's Excellence Strategy – The Berlin Mathematics Research Center MATH$^+$ (EXC-2046/1, project ID 390685689, BMS Stipend).
Mathieu Besançon was partially supported by MIAI at Université Grenoble Alpes (ANR-19-P3IA-0003).


\bibliography{bibliography}

\begin{thebibliography}{}

\bibitem[Bach, 2021]{bach2021effectiveness}
Bach, F. (2021).
\newblock On the effectiveness of {R}ichardson extrapolation in data science.
\newblock {\em SIAM Journal on Mathematics of Data Science}, 3(4):1251--1277.

\bibitem[Bach et~al., 2012]{bach2012equivalence}
Bach, F., Lacoste-Julien, S., and Obozinski, G. (2012).
\newblock On the equivalence between herding and conditional gradient
  algorithms.
\newblock In {\em Proceedings of the International Conference on Machine
  Learning}, pages 1355--1362. PMLR.

\bibitem[Beck and Shtern, 2017]{beck2017linearly}
Beck, A. and Shtern, S. (2017).
\newblock Linearly convergent away-step conditional gradient for non-strongly
  convex functions.
\newblock {\em Mathematical Programming}, 164:1--27.

\bibitem[Bertsimas and Tsitsiklis, 1997]{bertsimas1997introduction}
Bertsimas, D. and Tsitsiklis, J.~N. (1997).
\newblock {\em Introduction to linear optimization}, volume~6.
\newblock Athena scientific Belmont, MA.

\bibitem[Besan{\c{c}}on et~al., 2022]{besanccon2022frankwolfe}
Besan{\c{c}}on, M., Carderera, A., and Pokutta, S. (2022).
\newblock Frankwolfe. jl: A high-performance and flexible toolbox for
  {F}rank--{W}olfe algorithms and conditional gradients.
\newblock {\em INFORMS Journal on Computing}, 34(5):2611--2620.

\bibitem[Bettiol et~al., 2024]{bettiol2024oracle}
Bettiol, E., Buchheim, C., De~Santis, M., and Rinaldi, F. (2024).
\newblock An oracle-based framework for robust combinatorial optimization.
\newblock {\em Journal of Global Optimization}, 88(1):27--51.

\bibitem[Bezanson et~al., 2017]{bezanson2017julia}
Bezanson, J., Edelman, A., Karpinski, S., and Shah, V.~B. (2017).
\newblock Julia: A fresh approach to numerical computing.
\newblock {\em SIAM review}, 59(1):65--98.

\bibitem[Bomze et~al., 2019]{bomze2019first}
Bomze, I.~M., Rinaldi, F., and Bulo, S.~R. (2019).
\newblock First-order methods for the impatient: Support identification in
  finite time with convergent {F}rank--{W}olfe variants.
\newblock {\em SIAM Journal on Optimization}, 29(3):2211--2226.

\bibitem[Bomze et~al., 2020]{bomze2020active}
Bomze, I.~M., Rinaldi, F., and Zeffiro, D. (2020).
\newblock Active set complexity of the away-step {F}rank--{W}olfe algorithm.
\newblock {\em SIAM Journal on Optimization}, 30(3):2470--2500.

\bibitem[Braun et~al., 2022]{braun2022conditional}
Braun, G., Carderera, A., Combettes, C.~W., Hassani, H., Karbasi, A., Mokhtari,
  A., and Pokutta, S. (2022).
\newblock Conditional gradient methods.
\newblock {\em arXiv preprint arXiv:2211.14103}.

\bibitem[Braun et~al., 2019]{braun2019blended}
Braun, G., Pokutta, S., Tu, D., and Wright, S. (2019).
\newblock Blended conditonal gradients.
\newblock In {\em Proceedings of the International Conference on Machine
  Learning}, pages 735--743. PMLR.

\bibitem[Braun et~al., 2017]{braun2017lazifying}
Braun, G., Pokutta, S., and Zink, D. (2017).
\newblock Lazifying conditional gradient algorithms.
\newblock In {\em Proceedings of the International Conference on Machine
  Learning}, pages 566--575. PMLR.

\bibitem[Bugg et~al., 2022]{bugg2022nonnegative}
Bugg, C.~X., Chen, C., and Aswani, A. (2022).
\newblock Nonnegative tensor completion via integer optimization.
\newblock In Oh, A.~H., Agarwal, A., Belgrave, D., and Cho, K., editors, {\em
  Proceedings of Advances in Neural Information Processing Systems}.

\bibitem[Cai and Zhang, 2013]{cai2013sparse}
Cai, T.~T. and Zhang, A. (2013).
\newblock Sparse representation of a polytope and recovery of sparse signals
  and low-rank matrices.
\newblock {\em IEEE transactions on information theory}, 60(1):122--132.

\bibitem[Carath{\'e}odory, 1907]{caratheodory1907variabilitatsbereich}
Carath{\'e}odory, C. (1907).
\newblock {\"U}ber den variabilit{\"a}tsbereich der koeffizienten von
  potenzreihen, die gegebene werte nicht annehmen.
\newblock {\em Mathematische Annalen}, 64(1):95--115.

\bibitem[Carderera et~al., 2021]{carderera2021cindy}
Carderera, A., Pokutta, S., Sch{\"u}tte, C., and Weiser, M. (2021).
\newblock {CINDy}: Conditional gradient-based identification of non-linear
  dynamics--noise-robust recovery.
\newblock {\em arXiv preprint arXiv:2101.02630}.

\bibitem[Combettes and Pokutta, 2020]{combettes2020boosting}
Combettes, C. and Pokutta, S. (2020).
\newblock Boosting {F}rank-{W}olfe by chasing gradients.
\newblock In {\em Proceedings of the International Conference on Machine
  Learning}, pages 2111--2121. PMLR.

\bibitem[Combettes and Pokutta, 2023]{combettes2023revisiting}
Combettes, C.~W. and Pokutta, S. (2023).
\newblock Revisiting the approximate {C}arath{\'e}odory problem via the
  {F}rank-{W}olfe algorithm.
\newblock {\em Mathematical Programming}, 197(1):191--214.

\bibitem[Condat, 2016]{condat2016fast}
Condat, L. (2016).
\newblock Fast projection onto the simplex and the l1 ball.
\newblock {\em Mathematical Programming}, 158(1-2):575--585.

\bibitem[Davis, 2004]{umfpack}
Davis, T.~A. (2004).
\newblock Algorithm 832: Umfpack v4.3---an unsymmetric-pattern multifrontal
  method.
\newblock {\em ACM Trans. Math. Softw.}, 30(2):196--199.

\bibitem[Dunn and Harshbarger, 1978]{dunn1978conditional}
Dunn, J.~C. and Harshbarger, S. (1978).
\newblock Conditional gradient algorithms with open loop step size rules.
\newblock {\em Journal of Mathematical Analysis and Applications},
  62(2):432--444.

\bibitem[Filippozzi et~al., 2023]{filippozzi2023first}
Filippozzi, R., Gon{\c{c}}alves, D.~S., and Santos, L.-R. (2023).
\newblock First-order methods for the convex hull membership problem.
\newblock {\em European Journal of Operational Research}, 306(1):17--33.

\bibitem[Frank and Wolfe, 1956]{frank1956algorithm}
Frank, M. and Wolfe, P. (1956).
\newblock An algorithm for quadratic programming.
\newblock {\em Naval Research Logistics Quarterly}, 3(1-2):95--110.

\bibitem[Garber and Meshi, 2016]{garber2016linear}
Garber, D. and Meshi, O. (2016).
\newblock Linear-memory and decomposition-invariant linearly convergent
  conditional gradient algorithm for structured polytopes.
\newblock In {\em Proceedings of the International Conference on Neural
  Information Processing Systems}, pages 1009--1017. PMLR.

\bibitem[Garber and Wolf, 2021]{garber2021frank}
Garber, D. and Wolf, N. (2021).
\newblock {F}rank-{W}olfe with a nearest extreme point oracle.
\newblock In {\em Proceedings of International Conference on Learning Theory},
  pages 2103--2132. PMLR.

\bibitem[Gill et~al., 1987]{gill1987maintaining}
Gill, P.~E., Murray, W., Saunders, M.~A., and Wright, M.~H. (1987).
\newblock Maintaining {LU} factors of a general sparse matrix.
\newblock {\em Linear Algebra and its Applications}, 88:239--270.

\bibitem[Goldfarb et~al., 2017]{goldfarb2017linear}
Goldfarb, D., Iyengar, G., and Zhou, C. (2017).
\newblock Linear convergence of stochastic {F}rank-{W}olfe variants.
\newblock In {\em Proceedings of the International Conference on Artificial
  Intelligence and Statistics}, pages 1066--1074. PMLR.

\bibitem[Gu{\'e}lat and Marcotte, 1986]{guelat1986some}
Gu{\'e}lat, J. and Marcotte, P. (1986).
\newblock Some comments on {W}olfe's `away step'.
\newblock {\em Mathematical Programming}, 35(1):110--119.

\bibitem[Guo et~al., 2017]{guo2017efficient}
Guo, X., Yao, Q., and Kwok, J. (2017).
\newblock Efficient sparse low-rank tensor completion using the {F}rank-{W}olfe
  algorithm.
\newblock In {\em Proceedings of the AAAI Conference on Artificial
  Intelligence}, volume~31.

\bibitem[Guyon et~al., 2004]{guyon2004result}
Guyon, I., Gunn, S., Ben-Hur, A., and Dror, G. (2004).
\newblock Result analysis of the nips 2003 feature selection challenge.
\newblock {\em Proceedings of Advances in Neural Information Processing
  Systems}, 17.

\bibitem[Holloway, 1974]{holloway1974extension}
Holloway, C.~A. (1974).
\newblock An extension of the {F}rank and {W}olfe method of feasible
  directions.
\newblock {\em Mathematical Programming}, 6(1):14--27.

\bibitem[Huangfu and Hall, 2015]{huangfu2015novel}
Huangfu, Q. and Hall, J.~J. (2015).
\newblock Novel update techniques for the revised simplex method.
\newblock {\em Computational Optimization and Applications}, 60:587--608.

\bibitem[Jaggi, 2013]{jaggi2013revisiting}
Jaggi, M. (2013).
\newblock Revisiting {F}rank-{W}olfe: Projection-free sparse convex
  optimization.
\newblock In {\em Proceedings of the International Conference on Machine
  Learning}, pages 427--435. PMLR.

\bibitem[Kerdreux et~al., 2021]{kerdreux2021affine}
Kerdreux, T., Liu, L., Lacoste-Julien, S., and Scieur, D. (2021).
\newblock Affine invariant analysis of {F}rank-{W}olfe on strongly convex sets.
\newblock In {\em Proceedings of the International Conference on Machine
  Learning}, pages 5398--5408. PMLR.

\bibitem[Lacoste-Julien and Jaggi, 2013]{lacoste2013affine}
Lacoste-Julien, S. and Jaggi, M. (2013).
\newblock An affine invariant linear convergence analysis for {F}rank-{W}olfe
  algorithms.
\newblock {\em arXiv preprint arXiv:1312.7864}.

\bibitem[Lacoste-Julien and Jaggi, 2015]{lacoste2015global}
Lacoste-Julien, S. and Jaggi, M. (2015).
\newblock On the global linear convergence of {F}rank-{W}olfe optimization
  variants.
\newblock In {\em Proceedings of Advances in Neural Information Processing
  Systems}, pages 496--504.

\bibitem[Levitin and Polyak, 1966]{levitin1966constrained}
Levitin, E.~S. and Polyak, B.~T. (1966).
\newblock Constrained minimization methods.
\newblock {\em USSR Computational Mathematics and Mathematical Physics},
  6(5):1--50.

\bibitem[Macdonald et~al., 2022]{macdonald2022interpretable}
Macdonald, J., Besan{\c{c}}on, M.~E., and Pokutta, S. (2022).
\newblock Interpretable neural networks with {F}rank-{W}olfe: Sparse relevance
  maps and relevance orderings.
\newblock In {\em Proceedings of the International Conference on Machine
  Learning}, pages 14699--14716. PMLR.

\bibitem[Mu et~al., 2016]{mu2016scalable}
Mu, C., Zhang, Y., Wright, J., and Goldfarb, D. (2016).
\newblock Scalable robust matrix recovery: {F}rank--{W}olfe meets proximal
  methods.
\newblock {\em SIAM Journal on Scientific Computing}, 38(5):A3291--A3317.

\bibitem[Pedregosa et~al., 2018]{pedregosa2018step}
Pedregosa, F., Askari, A., Negiar, G., and Jaggi, M. (2018).
\newblock Step-size adaptivity in projection-free optimization.
\newblock {\em arXiv preprint arXiv:1806.05123}.

\bibitem[Pena and Rodriguez, 2019]{pena2019polytope}
Pena, J. and Rodriguez, D. (2019).
\newblock Polytope conditioning and linear convergence of the {F}rank--{W}olfe
  algorithm.
\newblock {\em Mathematics of Operations Research}, 44(1):1--18.

\bibitem[Pe{\~n}a, 2023]{pena2023affine}
Pe{\~n}a, J.~F. (2023).
\newblock Affine invariant convergence rates of the conditional gradient
  method.
\newblock {\em SIAM Journal on Optimization}, 33(4):2654--2674.

\bibitem[Pokutta, 2024]{P2023}
Pokutta, S. (2024).
\newblock {The Frank-Wolfe algorithm: a short introduction}.
\newblock {\em {Jahresbericht der Deutschen Mathematiker-Vereinigung}},
  126:3--35.

\bibitem[Pokutta et~al., 2020]{pokutta2020deep}
Pokutta, S., Spiegel, C., and Zimmer, M. (2020).
\newblock Deep neural network training with {F}rank-{W}olfe.
\newblock {\em arXiv preprint arXiv:2010.07243}.

\bibitem[Rao et~al., 2015]{rao2015forward}
Rao, N., Shah, P., and Wright, S. (2015).
\newblock Forward--backward greedy algorithms for atomic norm regularization.
\newblock {\em IEEE Transactions on Signal Processing}, 63(21):5798--5811.

\bibitem[Schork and Gondzio, 2017]{schork2017permuting}
Schork, L. and Gondzio, J. (2017).
\newblock Permuting spiked matrices to triangular form and its application to
  the {F}orrest-{T}omlin update.
\newblock Technical report, Technical Report ERGO-17-002, University of
  Edinburgh.

\bibitem[Tsuji et~al., 2022]{tsuji2022pairwise}
Tsuji, K.~K., Tanaka, K., and Pokutta, S. (2022).
\newblock Pairwise conditional gradients without swap steps and sparser kernel
  herding.
\newblock In {\em Proceedings of the International Conference on Machine
  Learning}, pages 21864--21883. PMLR.

\bibitem[Wirth et~al., 2023a]{wirth2023approximate}
Wirth, E., Kera, H., and Pokutta, S. (2023a).
\newblock Approximate vanishing ideal computations at scale.
\newblock In {\em Proceedings of the International Conference on Learning
  Representations}.

\bibitem[Wirth et~al., 2023b]{wirth2023acceleration}
Wirth, E., Kerdreux, T., and Pokutta, S. (2023b).
\newblock Acceleration of {F}rank-{W}olfe algorithms with open-loop step-sizes.
\newblock In {\em Proceedings of the International Conference on Artificial
  Intelligence and Statistics}, pages 77--100. PMLR.

\bibitem[Wirth et~al., 2023c]{wirth2023accelerated}
Wirth, E., Pena, J., and Pokutta, S. (2023c).
\newblock Accelerated affine-invariant convergence rates of the {F}rank-{W}olfe
  algorithm with open-loop step-sizes.
\newblock {\em arXiv preprint arXiv:2310.04096}.

\bibitem[Wirth and Pokutta, 2022]{wirth2022conditional}
Wirth, E. and Pokutta, S. (2022).
\newblock Conditional gradients for the approximately vanishing ideal.
\newblock In {\em Proceedings of the International Conference on Artificial
  Intelligence and Statistics}, pages 2191--2209. PMLR.

\bibitem[Wolfe, 1970]{wolfe1970convergence}
Wolfe, P. (1970).
\newblock Convergence theory in nonlinear programming.
\newblock {\em Integer and Nonlinear Programming}, pages 1--36.

\end{thebibliography}





\clearpage

\onecolumn

\section*{Checklist}

%

\begin{enumerate}

	\item For all models and algorithms presented, check if you include:
	\begin{enumerate}
		\item A clear description of the mathematical setting, assumptions, algorithm, and/or model. [Yes/No/Not Applicable]
		Yes, all relevant properties of our algorithms are discussed and presented rigorously. See Sections~\ref{sec:preliminaries}--\ref{sec:algorithmic}.
		\item An analysis of the properties and complexity (time, space, sample size) of any algorithm. [Yes/No/Not Applicable]
		Yes. See Sections~\ref{sec:preliminaries}--\ref{sec:algorithmic}.
		\item (Optional) Anonymized source code, with specification of all dependencies, including external libraries. [Yes/No/Not Applicable]
		{Yes.}
	\end{enumerate}

	\item For any theoretical claim, check if you include:
	\begin{enumerate}
		\item Statements of the full set of assumptions of all theoretical results. [Yes/No/Not Applicable]
		Yes, our theoretical results strive to be self-contained. See Sections~\ref{sec:preliminaries}--\ref{sec:algorithmic}.
		\item Complete proofs of all theoretical results. [Yes/No/Not Applicable]
		Yes, see Appendix~\ref{sec.miss_proof}. 
		\item Clear explanations of any assumptions. [Yes/No/Not Applicable] 
		Yes. See Sections~\ref{sec:preliminaries}--\ref{sec:algorithmic}.
	\end{enumerate}

	\item For all figures and tables that present empirical results, check if you include:
	\begin{enumerate}
		\item The code, data, and instructions needed to reproduce the main experimental results (either in the supplemental material or as a URL). [Yes/No/Not Applicable]
		{The instructions are in the README of the supplementary directory.}
		\item All the training details (e.g., data splits, hyperparameters, how they were chosen). [Yes/No/Not Applicable]
		{Yes, everything available in the supplementary materials.}
		\item A clear definition of the specific measure or statistics and error bars (e.g., with respect to the random seed after running experiments multiple times). [Yes/No/Not Applicable] {Not Applicable}
		\item A description of the computing infrastructure used. (e.g., type of GPUs, internal cluster, or cloud provider). [Yes/No/Not Applicable]
		{Yes, beginning of Section 11.}
	\end{enumerate}
	
	\item If you are using existing assets (e.g., code, data, models) or curating/releasing new assets, check if you include:
	\begin{enumerate}
		\item Citations of the creator If your work uses existing assets. [Yes/No/Not Applicable]
		Yes.
		\item The license information of the assets, if applicable. [Yes/No/Not Applicable]
		Not applicable.
		\item New assets either in the supplemental material or as a URL, if applicable. [Yes/No/Not Applicable]
		Not applicable.
		\item Information about consent from data providers/curators. [Yes/No/Not Applicable]
		Not applicable.
		\item Discussion of sensible content if applicable, e.g., personally identifiable information or offensive content. [Yes/No/Not Applicable]
		Not applicable.
	\end{enumerate}
	
	\item If you used crowdsourcing or conducted research with human subjects, check if you include:
	\begin{enumerate}
		\item The full text of instructions given to participants and screenshots. [Yes/No/Not Applicable]
		Not applicable.
		\item Descriptions of potential participant risks, with links to Institutional Review Board (IRB) approvals if applicable. [Yes/No/Not Applicable]
		Not applicable.
		\item The estimated hourly wage paid to participants and the total amount spent on participant compensation. [Yes/No/Not Applicable]
		Not applicable.
	\end{enumerate}
	
\end{enumerate}

\clearpage

\section{Missing proofs}\label{sec.miss_proof}

\begin{proof}[Proof of Theorem~\ref{thm:fw_sublinear}]
    The convergence result does not depend on any properties of convex combinations \citep{jaggi2013revisiting}.
\end{proof}
\begin{proof}[Proof of Theorem~\ref{thm:afw_bpfw_sublinear}]
    The proof in \citet{tsuji2022pairwise} does not require any information on the convex combination except when determining an upper bound on the number of so-called drop steps, steps which drop certain vertices from the active set. Since the number of drop steps can only decrease when replacing the active set at iteration $t$ with a potentially smaller active set, the convergence rate guarantee is convex-combination-agnostic. The proof in \citet{tsuji2022pairwise} designed for \bpfw{} also applies to \afw{}.
\end{proof}
\begin{proof}[Proof of Theorem~\ref{thm:afw_bpfw_linear}]
    The argument is identical to the one in the proof of Theorem~\ref{thm:afw_bpfw_sublinear}.
\end{proof}

\begin{proof}[Proof of Proposition~\ref{prop:asc}]
The proof is organized as follows: We first prove Properties~\ref{property:asc_M_invertible}--\ref{property:asc_equality_S_T} depending on whether the if (Line~\ref{line:asc_if}) or the else (Line~\ref{line:asc_else}) clauses are executed. Then, we prove Properties~\ref{property:asc_M_to_S}--\ref{property:asc_S_subseteq_T}.
First, suppose that $\cD = \emptyset$.
Then:
\begin{enumerate}
    \item[\ref{property:asc_M_invertible}.] By Lines~\ref{line:asc_if_Q} and~\ref{line:asc_M}, $M$ is obtained from $Q = N$ via elementary column additions. By Assumption~\ref{assumpiton:asc_N_invertible}, Property~\ref{property:asc_M_invertible} is satisfied.
    \item[\ref{property:asc_lambda_i}.] Let $i\in[n+2]$ such that $\lambda_i > 0$. By Lines~\ref{line:asc_if_lambda} and~\ref{line:asc_if_Q}, there exists an $\ss\in\cT$ such that $\tilde{\ss} = Q_{:,i}$. By Line~\ref{line:asc_M}, $M_{:,i}=Q_{:,i}$. By Lines~\ref{line:asc_alpha} and~\ref{line:asc_S}, $\ss\in\cS$, proving Property~\ref{property:asc_lambda_i}.
    \item[\ref{property:asc_equality_S_T}.]
    By Lines~\ref{line:asc_if}--\ref{line:asc_if_Q} and Assumption~\ref{assumption:asc_T_to_N}, $Q\llambda = \tilde{\xx}$. By Lines~\ref{line:asc_M}--\ref{line:asc_S} and Property~\ref{property:asc_lambda_i},
    $\sum_{\ss\in \cS} \alpha_\ss \tilde{\ss} = M\llambda = Q\lambda = \tilde{\xx}$. Thus,  Property~\ref{property:asc_equality_S_T} is satisfied.
\end{enumerate}
Second, suppose that $\cD = \{\vv\}$, that is, we consider Line~\ref{line:asc_else}. By Assumption~\ref{assumption:asc_N_tO^{(t)}}, there does not exist an $i\in[n+2]$ such that $\tilde{\vv}=N_{:,i}$.
Let $P = \{\ggamma\in \R^{n+3} \mid A\ggamma = \tilde{\xx}, \ggamma \geq \zeros\} \subseteq [0, 1]^{n+3}$. By Lines~\ref{line:asc_else_v}--\ref{line:asc_else_mu}, $\tilde{\xx}=A\mmu$ and $\mmu$ is a feasible solution for the optimization problem
    \begin{align}\label{eq:lmp}
    \min \ & -\ee_{n+3}^\intercal \ggamma \\
    \text{subject to} \ & \ggamma\in P. \nonumber
\end{align}
Intuitively, Algorithm~\ref{alg:asc} performs one pivoting step akin to the simplex algorithm starting from the feasible solution $\mmu$.
We first prove that $\rr = -N^{-1}\tilde{\vv}$ has at least one negative entry. Suppose towards a contradiction that $\rr\geq \zeros$. Thus, 
$
    \dd =  \begin{pmatrix}
        \rr\\
        1
    \end{pmatrix}
    \geq \zeros$.
Since
$
    A\dd = -N\rr + \tilde{\vv}  =  N(-N^{-1}\tilde{\vv}) + \tilde{\vv} = \zeros,
$
the vector $\mmu + \theta \dd$ is infeasible only if one of its components is negative for some $\theta \geq 0$. Since $\mmu \geq \zeros$ and $\dd \geq \zeros$, we have $\mmu + \theta \dd \geq \zeros$ for all $\theta \geq 0$, implying that $P\subseteq[0,1]^{n+3}$ is unbounded, a contradiction.
Then:
\begin{enumerate}
    \item[\ref{property:asc_M_invertible}.]
    Since $\rr \not \geq \zeros$, it holds that
    $
        k \in \argmin_{i \in [n+2], r_i < 0}-\frac{\mu_i}{r_i}
    $
    as in Line~\ref{line:asc_else_k}
    exists. Thus, in Line~\ref{line:asc_else_Q}, Algorithm~\ref{alg:asc} constructs the matrix
    $
        Q = (N_{:, 1}, \ldots, N_{:, k-1}, \tilde{\vv}, N_{:, k+1}, \ldots, N_{:,n+2}) \in \R^{(n+2) \times (n+2)}.
    $
    The columns $N_{:, i}$, $i\neq k$, and $\tilde{\vv}$ are linearly independent \citep[Theorem~3.2]{bertsimas1997introduction}. Thus, by Assumption~\ref{assumpiton:asc_N_invertible} and Line~\ref{line:asc_else_Q}, $Q$ is invertible, $Q_{n+1, :} \geq \zeros^\intercal$, and $Q_{n+2, :} \geq \ones^\intercal$. 
    By Line~\ref{line:asc_M}, $M$ is obtained from $Q$ via elementary column additions. Thus, Property~\ref{property:asc_M_invertible} is satisfied.
    \item[\ref{property:asc_lambda_i}.] Let $i\in[n+2]$ such that $\lambda_i > 0$. By \citet[Theorem~3.2]{bertsimas1997introduction}, $Q\llambda =\tilde{\xx}$ and $\llambda\geq\zeros$. Since $\tilde{x}_{n+1} = 0$, $Q_{n+1, :}\geq \zeros^\intercal$, and $\lambda_i > 0$, we have $Q_{n+1, i}=0$. By Line~\ref{line:asc_else_Q} and Assumption~\ref{assumption:asc_N_tO^{(t)}}, there exists an $\ss\in\cT$ such that $\tilde{\ss}= Q_{:,i}$. By Line~\ref{line:asc_M}, $M_{:,i}=Q_{:,i}$. By Lines~\ref{line:asc_alpha} and~\ref{line:asc_S}, $\ss\in\cS$, proving Property~\ref{property:asc_lambda_i}.
    \item[\ref{property:asc_equality_S_T}.]
    We have already established that $Q\llambda = \tilde{\xx}$. By Lines~\ref{line:asc_M}--\ref{line:asc_S} and Property~\ref{property:asc_lambda_i},
    $\sum_{\ss\in \cS} \alpha_\ss \tilde{\ss} = M\llambda = Q\lambda = \tilde{\xx}$. Thus,  Property~\ref{property:asc_equality_S_T} is satisfied.
\end{enumerate}
Finally, we prove Properties~\ref{property:asc_M_to_S}--\ref{property:asc_S_subseteq_T} irrespective of whether the if (Line~\ref{line:asc_if}) or the else (Line~\ref{line:asc_else}) clauses are executed:
\begin{enumerate}
    \item[\ref{property:asc_M_to_S}.] Let $i\in[n+2]$ such that $M_{n+1,i} = 0$. By Line~\ref{line:asc_M}, $\lambda_i > 0$.
    By Property~\ref{property:asc_lambda_i}, there exists an $\ss\in\cS$ such that $\tilde{\ss} = M_{:, i}$. Thus, Property~\ref{property:asc_M_to_S} is satisfied.
    \item[\ref{property:asc_S_to_M}.] Let $\ss\in\cS$. By Line~\ref{line:asc_S}, $\alpha_\ss > 0$. By Line~\ref{line:asc_alpha}, there exists an $i \in[n+2]$ such that $M_{:,i} = \tilde{\ss}$. Thus, Property~\ref{property:asc_S_to_M} is satisfied. 
    \item[\ref{property:asc_S_subseteq_T}.]
    Let $\ss \in \cS$. By Line~\ref{line:asc_S}, $\alpha_\ss > 0$. By Line~\ref{line:asc_alpha}, there exists an $i\in[n+2]$ such that $\tilde{\ss} = M_{:,i}$ and $\lambda_i > 0$. By Property~\ref{property:asc_lambda_i}, $\ss\in\cT$. Thus, $\cS\subseteq\cT$. By Property~\ref{property:asc_M_invertible}, $M$ is invertible. Thus, there exists an $i\in[n+2]$ such that $M_{n+1,i} \neq 0$. Thus, for all $\ss\in\cS$, $\tilde{\ss}\neq M_{:,i}$. By Property~\ref{property:asc_S_to_M}, $|\cS| \leq n+1$. Thus, Property~\ref{property:asc_S_subseteq_T} is satisfied.
\end{enumerate}
\end{proof}

\begin{proof}[Proof of Theorem~\ref{thm:pm}]
    Properties~\ref{property:pm_convex_combination}--\ref{property:pm_active_set_size_reduction} follow from Proposition~\ref{prop:asc} and induction. Since \pm modifies the \caa as described in Definition~\ref{def:cca}, \pm{} conserves convex-combination-agnostic properties of the corresponding \caa{}.
\end{proof}

\begin{proof}[Proof of Theorem~\ref{thm:active_set_pm}]
    Let $t\in \{\Rfwt, \Rfwt+1, \ldots, T\}$ and consider the convex combination representing the current iterate $\xx^{(t)} = \sum_{s\in\cS^{(t)}} \alpha_\ss^{(t)} \ss$, where $\aalpha^{(t)}\in \Delta_{|\cV|}$ and $\cS^{(t)} = \{\ss\in\cV\mid \alpha_\ss^{(t)} > 0\}$. Suppose towards a contradiction that $|\cS^{(t)}| \geq \dim(\cC^*) + 2$.
    By Theorem~\ref{thm:pm}, $M^{(t)}$ is invertible for all $t\in\{0,1 \ldots, T\}$. Thus, for any subset $\{\pp_1, \pp_2, \ldots, \pp_{\dim(\cC^*)+2}\} \subseteq \cS^{(t)}$, it has to hold that $\tilde{\pp}_1, \tilde{\pp}_2, \ldots, \tilde{\pp}_{\dim(\cC^*)+2}$ are linearly independent. Since $\dim(\cC^*) + 2$ vectors of dimension $\dim(\cC^*) + 1$ cannot be linearly independent, we have a contradiction.
\end{proof}

\begin{proof}[Proof of Theorem~\ref{thm:afw_active}]
The statement of Theorem~\ref{thm:afw_active} without the convex-combination-agnostic property is a simplified version derived from \citet[Theorem~C.1]{bomze2020active}. The proof of Theorem~C.1 is a reduction to \citet[Theorem~4.3]{bomze2020active} that does not utilize any arguments involving the active set of \afw. Furthermore, it can be verified that \citet[Theorem~4.3]{bomze2020active} and all the results it builds upon are convex-combination-agnostic. Thus, the convex-combination-agnostic property of the statement in Theorem~\ref{thm:afw_active} is established.
\end{proof}

\begin{proof}[Proof of Corollary~\ref{corollary:asi_pm_afw}]
The result follows by combining Theorem~\ref{thm:active_set_pm} with Theorem~\ref{thm:afw_active} for \afw, and with Theorem~\ref{thm:bpfw_active} for \bpfw, and noting that $\xx^{(t)}\in\cC^*$ for $\cC^*\in\faces(\cC)$ implies $\cS^{(t)}\subseteq \cC^*$.
\end{proof}

\section{\fw{} variants}\label{sec:fw_variants}

The away-step Frank-Wolfe algorithm (\afw{}) \citep{wolfe1970convergence, guelat1986some,lacoste2015global} is presented in Algorithm~\ref{alg:afw}.
The blended pairwise Frank-Wolfe algorithm (\bpfw{}) \citep{tsuji2022pairwise} is presented in Algorithm~\ref{alg:bpfw}.

\begin{algorithm}[t]
\SetKwInput{Input}{Input} \SetKwInput{Output}{Output}
\SetKwComment{Comment}{$\triangleright$\ }{}
\caption{Away-step Frank-Wolfe algorithm (\afw{}) with line-search}\label{alg:afw}
\Input{$\xx^{(0)}\in\cV$.
} 
  \Output{$\aalpha^{(T)}\in\Delta_{|\cV|}$, $\cS^{(T)} = \{\ss\in \cV\mid \alpha_\ss^{(T)} > 0\}$, and $\xx^{(T)}\in\cC$ such that $\xx^{(T)} = \sum_{\ss\in \cS^{(T)}} \alpha_\ss^{(T)} \ss$.}
\hrulealg
$\aalpha^{(0)}\in\Delta_{|\cV|}$ s.t. for all $\ss\in\cV$,  $\alpha_\ss^{(0)} \gets\begin{cases}
            1,& \text{if} \ \ss = \xx^{(0)}\\
            0, & \text{if} \ \ss\in\cV\setminus\{\xx^{(0)}\}
            \end{cases}$\\
            $\cS^{(0)} \gets \{\ss \in \cV \mid \alpha_\ss^{(0)} > 0\} = \{\xx^{(0)}\}$\\
\For{$t = 0,\ldots, T - 1$}{
{$\aa^{(t)} \gets \argmax_{\vv\in \cS^{(t)}} \langle \nabla f(\xx^{(t)}), \vv \rangle$ \Comment*[f]{away vertex}}\\
{$\vv^{(t)} \gets \argmin_{\vv\in \cV} \langle \nabla f(\xx^{(t)}), \vv \rangle$ \Comment*[f]{\fw{} vertex}}\\
\uIf{$\langle \nabla f(\xx^{(t)}),\vv^{(t)} - \xx^{(t)}\rangle \geq  \langle \nabla f(\xx^{(t)}), \xx^{(t)} - \aa^{(t)}\rangle$}{
    {$\dd^{(t)} \gets \xx^{(t)} - \aa^{(t)}$ \Comment*[f]{away step}\label{line:afw_away_step}}\\
    {$\eta^{(t)} \gets \argmin_{\eta\in [0, \alpha^{(t)}_{\aa^{(t)}}/(1- \alpha^{(t)}_{\aa^{(t)}})]} f(\xx^{(t)} + \eta \dd^{(t)})$}\\
    {$\aalpha^{(t+1)}\in\Delta_{|\cV|}$ s.t. for all $\ss\in\cV$,  $\alpha_\ss^{(t+1)} \gets \begin{cases}
    (1+\eta^{(t)})\alpha_\ss^{(t)}, & \text{if} \ \ss\in\cV\setminus\{\aa^{(t)}\}\\
    (1+\eta^{(t)})\alpha_\ss^{(t)} - \eta^{(t)}, & \text{if} \ \ss=\aa^{(t)}\\
    \end{cases}
    $
    }\\
    }
\Else{
{$\dd^{(t)} \gets \vv^{(t)} - \xx^{(t)}$ \Comment*[f]{\fw{} step}}\\
{$\eta^{(t)} \gets \argmin_{\eta \in [0, 1]} f(\xx^{(t)} + \eta \dd^{(t)})$ }\\
{$\aalpha^{(t+1)}\in\Delta_{|\cV|}$ s.t. for all $\ss\in\cV$,  $\alpha_\ss^{(t+1)} \gets \begin{cases}
(1-\eta^{(t)})\alpha_\ss^{(t)}, & \text{if} \ \ss\in \cV\setminus \{\vv^{(t)}\}\\
(1-\eta^{(t)})\alpha_\ss^{(t)} + \eta^{(t)}, & \text{if} \ \ss = \vv^{(t)}\\
\end{cases}
$
}
}
{$\cS^{(t+1)} \gets \{\ss\in \cV\mid  \alpha_{\ss}^{(t+1)}>0\}$}\\
{$\xx^{(t+1)} \gets \xx^{(t)} + \eta^{(t)} \dd^{(t)}$}
}
\end{algorithm}

\begin{algorithm}[t]
\SetKwInput{Input}{Input} \SetKwInput{Output}{Output}
\SetKwComment{Comment}{$\triangleright$\ }{}
\caption{Blended pairwise Frank-Wolfe algorithm (\bpfw{}) with line-search}\label{alg:bpfw}
\Input{$\xx^{(0)}\in\cV$.
} 
  \Output{$\aalpha^{(T)}\in\Delta_{|\cV|}$, $\cS^{(T)} = \{\ss\in \cV\mid \alpha_\ss^{(T)} > 0\}$, and $\xx^{(T)}\in\cC$, such that $\xx^{(T)} = \sum_{\ss\in \cS^{(T)}} \alpha_\ss^{(T)} \ss$.}
\hrulealg
  $\aalpha^{(0)}\in\Delta_{|\cV|}$ s.t. for all $\ss\in\cV$,  $\alpha_\ss^{(0)} \gets\begin{cases}
            1,& \text{if} \ \ss = \xx^{(0)}\\
            0, & \text{if} \ \ss\in\cV\setminus\{\xx^{(0)}\}
            \end{cases}$\\
            $\cS^{(0)} \gets \{\ss \in \cV \mid \alpha_\ss^{(0)} > 0\} = \{\xx^{(0)}\}$\\
\For{$t = 0,\ldots, T - 1$}{
{$\aa^{(t)} \gets \argmax_{\vv\in \cS^{(t)}} \langle \nabla f(\xx^{(t)}), \vv \rangle$ \Comment*[f]{away vertex}}\\
{$\ww^{(t)} \gets \argmin_{\vv\in \cS^{(t)}} \langle \nabla f(\xx^{(t)}), \vv \rangle$ \Comment*[f]{local \fw{} vertex}}\\
{$\vv^{(t)} \gets \argmin_{\vv\in \cV} \langle \nabla f(\xx^{(t)}), \vv \rangle$ \Comment*[f]{\fw{} vertex}}\\
\uIf{$\langle \nabla f(\xx^{(t)}),\vv^{(t)} - \xx^{(t)}\rangle \geq  \langle \nabla f(\xx^{(t)}), \ww^{(t)} - \aa^{(t)}\rangle$}{
    {$\dd^{(t)} \gets \ww^{(t)} - \aa^{(t)}$ \Comment*[f]{local pairwise step}\label{line:bpfw_local_pairwise_step}}\\
    {$\eta^{(t)} \gets \argmin_{\eta\in [0, \alpha_{\aa^{(t)}}^{(t)}]} f(\xx^{(t)} + \eta \dd^{(t)})$}\\
    {$\aalpha^{(t+1)}\in\Delta_{|\cV|}$ s.t. for all $\ss\in\cV$,  $\alpha_\ss^{(t+1)} \gets \begin{cases}
    \alpha_\ss^{(t)}, & \text{if} \ \ss\in\cV\setminus\{\aa^{(t)}, \ww^{(t)}\}\\
    \alpha_\ss^{(t)} -\eta^{(t)}, & \text{if} \ \ss=\aa^{(t)}\\
     \alpha_\ss^{(t)} +\eta^{(t)}, & \text{if} \ \ss=\ww^{(t)}
    \end{cases}
    $
    }\\
    }
\Else{
{$\dd^{(t)} \gets \vv^{(t)} - \xx^{(t)}$ \Comment*[f]{\fw{} step}}\\
{$\eta^{(t)} \gets \argmin_{\eta \in [0, 1]} f(\xx^{(t)} + \eta \dd^{(t)})$}\\
{$\aalpha^{(t+1)}\in\Delta_{|\cV|}$ s.t. for all $\ss\in\cV$,  $\alpha_\ss^{(t+1)} \gets \begin{cases}
(1-\eta^{(t)})\alpha_\ss^{(t)}, & \text{if} \ \ss\in \cV\setminus \{\vv^{(t)}\}\\
(1-\eta^{(t)})\alpha_\ss^{(t)} + \eta^{(t)}, & \text{if} \ \ss = \vv^{(t)}\\
\end{cases}
$
}
}
{$\cS^{(t+1)} \gets \{\ss\in \cV\mid  \alpha_{\ss}^{(t+1)}>0\}$}\\
{$\xx^{(t+1)} \gets \xx^{(t)} + \eta^{(t)} \dd^{(t)}$}
}
\end{algorithm}

\section{Active set identification with the blended pairwise Frank-Wolfe algorithm}\label{appendix:bpfw_asi}

In this section, we prove for \bpfw{} the active set identification property introduced in \citet{bomze2020active} and proved for \afw{}.
We consider in this section that the feasible set is the probability simplex $\Delta_{n}$. Generalization to polytopes follows by affine invariance \citep[Appendix C]{bomze2020active}.
This means in particular that each vertex is the unit vector $\ee_i$ corresponding to a coordinate $i \in [n]$.
For convenience, we introduce the set $\cS_I$ consisting of unit vectors corresponding to the coordinates for a given active set $\cS$:
\begin{align*}
    \cS_I = \{i \in [n] \mid \ee_i \in \cS\}.
\end{align*}

\noindent
We first introduce necessary definitions.
The multiplier associated with the $i$th coordinate at a point $\xx\in \Delta_{n}$ is:
\begin{align*}
\lambda_i(\xx) = \langle\nabla f(\xx), \ee_i - \xx\rangle,
\end{align*}
allowing the partition of the coordinates into
\begin{align}
    & I(\xx) = \{i\in [n] \mid \lambda_i(\xx) = 0 \} \qquad  \text{and} \qquad I^c(\xx) = [n] \,\backslash\, I(\xx).\label{def:Iset}
\end{align}
Given a stationary point of the problem $\xx^*$ and iterates $\xx^{(t)}$ generated by \bpfw{}, we also define:
\begin{align}
    & O^{(t)} = \{i\in I^c(\xx^*) \mid \xx_i^{(t)} = 0 \},\nonumber\\
    & J^{(t)} = I^c(\xx^*) \,\backslash\, O^{(t)},\nonumber\\
    & \delta^{(t)} = \max_{i\in [n] \,\backslash\, O^{(t)}} \lambda_i(\xx^*),\nonumber\\
    & \delta_{\min} = \min_{i\in I^c(\xx^*)} \lambda_i(\xx^*), \label{def:deltamin}\\
    & h^* = \frac{\delta_{\min}}{3L + \max\{\delta^{(t)}, \delta_{\min}\} }\nonumber .
\end{align}
By complementary slackness, $\xx_i^* = 0$ for all $i \in I^c(\xx^*)$.
We now introduce a first lemma restating the conditions for the selection of the type of step in \bpfw{} in terms of multipliers.
\begin{lemma}[\bpfw{} step multipliers]\label{lemma:nonintroductionmultipliers}
Let $\cC = \Delta_n\subseteq \R^n$ be the probability simplex and let $f\colon \cC \to \R$ be a convex and $L$-smooth function. Then, for the iterates of Algorithm~\ref{alg:bpfw} (\bpfw{}) with line-search and all $t\in \{0, 1, \ldots, T-1\}$, the following hold:
\begin{enumerate}
    \item If
    $
    \max_{j\in \cS_I^{(t)}} \lambda_j(\xx^{(t)}) - \min_{i\in \cS_I^{(t)}} \lambda_i(\xx^{(t)}) \geq - \min_{i\in [n]} \lambda_i(\xx^{(t)}),
    $
    then \bpfw{} performs a pairwise step, with $\dd^{(t)} = \ee_i - \ee_j$ for $i \in \argmin_{i \in \cS_I^{(t)}} \lambda_i(\xx^{(t)})$ and $j \in \argmax_{j\in \cS_I^{(t)}} \lambda_j(\xx^{(t)})$.
    \item \label{property:nonintroductionmultipliers_2} For every $j \in [n] \,\backslash\, \cS_I^{(t)}$, if $\lambda_j(\xx^{(t)}) > 0$, then $x^{(t+1)}_j = x^{(t)}_j = 0$.
\end{enumerate}
\end{lemma}
\begin{proof}
The first statement follows from the algorithm and the fact that the $i$th multiplier corresponds to the selection criterion for the away and \fw{} vertices.
For the second statement, note that a vertex not in the current active set cannot be selected as a \fw{} or away vertex in a pairwise step.
Furthermore, the \fw{} gap can be expressed as
\begin{align*}
    g(\xx) = -\min_{i\in [n]} \lambda_i(\xx) \geq 0.
\end{align*}
Thus, there always exists a coordinate $i$ with $\lambda_i(\xx) \leq 0$. If for $j \in [n]$, $\lambda_j(\xx^{(t)}) > 0$, then the vertex $\ee_j$ would not be selected as the vertex in a \fw{} step either.
\end{proof}
Next, we derive upper and lower bounds on the multipliers.
\begin{lemma}[Bounds on the multipliers]
Let $\cC = \Delta_n\subseteq \R^n$ be the probability simplex, let $f\colon \cC \to \R$ be a convex and $L$-smooth function, and let $\xx^*\in\argmin_{\xx\in\cC}f(\xx)$. Then, for the iterates of Algorithm~\ref{alg:bpfw} (\bpfw{}) with line-search and all $t\in \{0, 1, \ldots, T-1\}$ such that
$\norm{\xx^* - \xx^{(t)}}_1 < h^*$, it holds that:
\begin{align}
\lambda_i(\xx^{(t)}) &< h^* \left(L + \frac{\delta^{(t)}}{2}\right) && \mathrm{for\ all} \ i \in I(\xx^*)\label{ineq:multiplierInI} \\
\lambda_j(\xx^{(t)}) &> -h^* \left(L + \frac{\delta^{(t)}}{2}\right) + \delta_{\min} && \mathrm{for\ all} \ j \in I^c(\xx^*).\label{ineq:multiplierInJt}
\end{align}
\end{lemma}
\begin{proof}
\citet[Lemma 3.1]{bomze2020active} proves that
$
    |\lambda_i(\xx^*) - \lambda_i(\xx^{(t)})| < h^* \left(L + \frac{\delta^{(t)}}{2}\right)
$ for any $i\in[n]$ and iteration $t$ where the assumptions of the Lemma hold.
The two subsequent inequalities follow from the definitions of $I(\xx^*)$ in \eqref{def:Iset} and $\delta_{\min}$ in \eqref{def:deltamin}.
In particular, Equation~\eqref{ineq:multiplierInJt} can be obtained as follows:
\begin{align*}
\lambda_j(\xx^{(t)}) &> \lambda_j(\xx^{*}) - h^* \left(L + \frac{\delta^{(t)}}{2}\right) \geq \min_{i\in I^{c}(\xx^*)} \lambda_i(\xx^{*})  - h^* \left(L + \frac{\delta^{(t)}}{2}\right) = \delta_{\min} - h^* \left(L + \frac{\delta^{(t)}}{2}\right).
\end{align*}
\end{proof}
In order to prove the active set identification property, we need a lower bound assumption on the step size. This assumption allows us in particular to prove that, once the iterate is close enough to a stationary point, a local pairwise step is always maximal and reduces the weight of the away vertex down to zero.
\begin{assumption}[Step-size lower bound]\label{assum:steplowerbound}
During the execution of Algorithm~\ref{alg:bpfw} (\bpfw{}) and any iteration $t\in\{0,1,\ldots, T-1\}$ where a pairwise step is taken with $\aa^{(t)}\in\cS^{(t)}$ the away vertex, we assume that the following lower bound holds on the step size:
\begin{align*}
    \eta^{(t)} \geq \min \left\{\alpha_{\aa^{(t)}}^{(t)}, \frac{\langle\nabla f(\xx^{(t)}), - \dd^{(t)}\rangle}{L \norm{\dd^{(t)}}_2^2} \right\}.
\end{align*}
\end{assumption}
Assumption~\ref{assum:steplowerbound} holds in particular for step-sizes computed with line-search and short-step rules \citet[Remark 4.4]{bomze2020active}.
The authors also show that
\begin{align}
    \eta^{(t)} \geq \min \left\{\alpha_{\aa^{(t)}}^{(t)}, c \frac{\langle\nabla f(\xx^{(t)}), - \dd^{(t)}\rangle}{L \norm{\dd^{(t)}}_2^2} \right\}\label{eqref:relaxcond}
\end{align}
for a given $c > 0$ is a sufficient condition for Assumption~\ref{assum:steplowerbound} to hold (with a modified Lipschitz constant $L / c$).
This means in particular that the commonly-used adaptive step-size strategy introduced in \citet{pedregosa2018step} respects the weaker condition \eqref{eqref:relaxcond}, with the factor $c$ corresponding to the parameter $\tau$ of the backtracking subroutine \citet[Algorithm 2]{pedregosa2018step}.

We can now state Theorem~\ref{theorem:identification}, which proves the active set identification for \bpfw{}. The proof follows the structure of that of \citet[Theorem 3.3]{bomze2020active}, where the same property is shown for \afw{}.
\begin{theorem}[\bpfw{} identification]\label{theorem:identification}
Let $\cC = \Delta_n\subseteq \R^n$ be the probability simplex, let $f\colon \cC \to \R$ be a convex and $L$-smooth function, and let $\xx^*\in\argmin_{\xx\in\cC}f(\xx)$. Then, for the iterates of Algorithm~\ref{alg:bpfw} (\bpfw{}) with line-search and all $t\in \{0, 1, \ldots, T-1\}$ such that
$\norm{\xx^{(t)} - \xx^*}_1 < h^*$, it holds that
$
    |J^{(t+1)}| \leq \max \{0, |J^{(t)}| -1\}.
$
\end{theorem}
\begin{proof}
If $I^c(\xx^*) = \emptyset$, $J^{(t)} \subseteq I^c(\xx^*) = \emptyset$ and $J^{(t+1)} \subseteq I^c(\xx^*) = \emptyset$. We next assume $|I^c(\xx^*)| > 0$.

We first consider the case in which $J^{(t)} = \emptyset$.
Then, $[n] \, \backslash \, O^{(t)} = I(\xx^*)$ and $\delta^{(t)} = 0$ since $\lambda_i(\xx^*) = 0 $ for all $ i \in I(\xx^*)$.
Inequality \eqref{ineq:multiplierInJt} becomes: 
\begin{align*}
    \lambda_j(\xx^{(t)}) > -h^* L + \delta_{\min} = \delta_{\min} - \frac{L \delta_{\min}}{3L + \delta_{\min}} > 0 \qquad \text{for all} \ j \in I^c(\xx^*) = O^{(t)}.
\end{align*}
Since $J^{(t)} = \emptyset$, it holds that $I^c(\xx^*) = O^{(t)}$ and $x^{(t)}_j = 0 $ for all $j \in I^c(\xx^*)$.
Following Lemma~\ref{lemma:nonintroductionmultipliers}, Property~\ref{property:nonintroductionmultipliers_2}, we conclude that $x_j^{(t+1)} = 0 $ for all $j \in I^c(\xx^*)$. Thus, $J^{(t+1)} = \emptyset$.

We now consider the case when $|J^{(t)}| > 0$ and first prove that the \fw vertex $i$ and the away vertex $j$ are such that $i \in I(\xx^*)$ and $j\in J^{(t)}$.
Recall that $\cS^{(t)}_I \subseteq [n] \,\backslash\, O^{(t)} = J^{(t)} \cup I(\xx^*)$ 
and that both the \fw and away vertex are selected from $\cS^{(t)}_I$ in a pairwise step.
\citet[Theorem 3.3, Case 2]{bomze2020active}
proves that if $|J^{(t)}| > 0$, then the away vertex is always selected in $J^{(t)}$ if
\begin{align*}
    \|\xx^{(t)} - \xx^* \|_1 < \frac{\delta_{\min}}{2L + \delta_{\min}},
\end{align*}
which holds by the definition of $h^*$ since 
\begin{align*}
    \|\xx^{(t)} - \xx^* \|_1 < h^* \leq \frac{\delta_{\min}}{3L + \delta_{\min}} \leq \frac{\delta_{\min}}{2L + \delta_{\min}}.
\end{align*}
We next demonstrate that the \fw vertex is always selected in $I(\xx^*)$ and not in $J^{(t)}$.
Based on \eqref{ineq:multiplierInI} and \eqref{ineq:multiplierInJt}, a sufficient condition for the \fw vertex to be selected in $I(\xx^*)$ is:
\begin{align*}
& \lambda_i(\xx^{(t)}) < \lambda_j(\xx^{(t)}) \ \text{for all} \ i \in I(\xx^*), j \in J^{(t)} \qquad
\Leftrightarrow \qquad h^* < \frac{\delta_{\min}}{2L + \delta^{(t)}},
\end{align*}
which follows from the definition of $h^*$.
Finally, we need to prove that the pairwise step is maximal such that the away vertex is dropped. That is, given the \fw and away vertices $i \in I(\xx^*)$ and $j\in J^{(t)}$, respectively, the step-size is $\eta^{(t)} = \alpha_{\ee_j}^{(t)}$, the weight of the away vertex $\ee_j$ in the active set. 
Let us assume by contradiction that the step is not maximal, that is, $\eta^{(t)} < \alpha_{\ee_j}^{(t)}$.
Furthermore, recall that if a pairwise step is performed at iteration $t$, then $\dd^{(t)} = \ee_i - \ee_j$ for two indices $i\neq j$, yielding $\|\dd^{(t)}\|^2_2 = 2$.
Then, by Assumption~\ref{assum:steplowerbound},
\begin{align*}
    \eta^{(t)} & \geq \frac{\langle\nabla f(\xx^{(t)}), - \dd^{(t)}\rangle}{L \norm{\dd^{(t)}}_2^2} \\
               & = \frac{\lambda_j(\xx^{(t)}) - \lambda_i(\xx^{(t)})}{2L} & \\
               & \geq \frac{\delta_{\min} - h^*(2L+\delta^{(t)})}{2L} & \text{$\triangleright$ by subtracting \eqref{ineq:multiplierInI} from \eqref{ineq:multiplierInJt}}\\
               & = \frac{\delta_{\min}}{2L} - h^* - \frac{h^* \delta^{(t)}}{2L}.
\end{align*}
Furthermore, since $\xx^{(t)}$ and $\xx^*$ lie in the simplex, by \citet[Lemma A.1, Property 2]{bomze2020active}, and the fact that $x^*_j = 0$  for all $j\in I^c(\xx^*)$ by complementary slackness, we have:
\begin{align*}
x_j^{(t)} = |x_j^{(t)} - x_j^*| \leq \frac{\|\xx^{(t)} - \xx^*\|_1}{2} < \frac{h^*}2.
\end{align*}
Together with the fact that $d^{(t)}_j = -1$ since $j$ is the coordinate corresponding to the away vertex, we have:
\begin{align*}
x_j^{(t+1)} &= x_j^{(t)} + \eta^{(t)} d^{(t)}_j\\
            & < \frac{h^*}{2} - \frac{\delta_{\min}}{2L} + h^* + \frac{h^*\delta^{(t)}}{2L} \\
            & = \frac{3L h^*}{2L} - \frac{\delta_{\min}}{2L} + \frac{h^*\delta^{(t)}}{2L} \\
            & = \frac1{2L} ((3L + \delta^{(t)})h^* - \delta_{\min}) < 0,
\end{align*}
a contradiction. We conclude that the step size is maximal.
\end{proof}
Theorem~\ref{theorem:identification} can be used to generalize \citet[Theorem 4.3]{bomze2020active} to \bpfw{} under similar assumptions, meaning that \bpfw{} identifies the active set in a finite number of iterations,
since the proof of \citet[Theorem 4.3]{bomze2020active} does not rely on the sequence generated by the algorithm.
We also note that our proof does not extend to the pairwise Frank-Wolfe algorithm \citep{lacoste2015global}, since we rely on the \bpfw{} property that at any iteration in which a pairwise step is performed, both the away and the FW vertex are selected from the active set.

\begin{theorem}[Active set identification property of \bpfw]\label{thm:bpfw_active}
Let $\cC\subseteq\R^n$ be a polytope, let $f\colon\cC\to\R$ be a convex and $L$-smooth function, and suppose that $\xx^*\in\argmin_{\xx\in\cC}f(\xx)$ is unique and $\xx^*\in\relint(\cC^*)$, where $\cC^*\in\faces(\cC)$. Then, for $T$ large enough, for the iterations of Algorithm~\ref{alg:bpfw} (\bpfw) with line-search, there exists an iteration $\Rfwt \in \{0, 1, \ldots, T\}$ such that $\xx^{(t)}\in\cC^*$ for all $t\in\{\Rfwt, \Rfwt+1,\ldots, T\}$. This property is convex-combination-agnostic.
\end{theorem}
\begin{proof}
By affine invariance of \bpfw with line-search, this is equivalent to optimizing $f(A\yy)$, with $\yy\in\Delta_n$, as developed in \citet[Appendix~C]{bomze2020active}.
The proof over the simplex follows that of \citet[Theorem~4.3]{bomze2020active}, relying on Theorem~\ref{theorem:identification} to ensure the cardinality of the set $J^{(t)}$ decreases at each iteration once
there exists a stationary point $\xx^*$ of the problem with $\xx^* \in \cX$, such that $\|\xx^{(t)} - \xx^*\|_1 < h^*$.
\end{proof}

\section{Algorithmic aspects}\label{sec:algorithmic}

The computationally-demanding part of \pm{} -- compared to the baseline \fw{} variants it builds on -- amounts to solving the linear system $(N^{(t)})^{-1} \vv$ in \fwu{} when a new vertex is introduced.
All modifications of $M^{(t)}$ throughout \pm's execution consist of column additions and substitutions, one could therefore construct an initial sparse LU decomposition of $M^{(t)}$ and perform rank-one updates
\citep{gill1987maintaining,huangfu2015novel}.
Preliminary experiments using the \texttt{BasicLU} sparse LU library \citep{schork2017permuting} showed that the numerical accuracy was too low for our purpose,
resulting in weights $\lambda$ that presented a large iterate reconstruction error.
We resorted to direct sparse LU solves instead whenever a new vertex is inserted in the matrix without maintaining an LU update.
Furthermore, we compensate for inaccuracy of the linear solver by projecting back the weights $\lambda$ onto $\Delta_{n+2}$ using the algorithm proposed in \citet{condat2016fast}.
As observed in the experiments, the per-iteration overhead caused by the pivoting operations is not deterrent enough to make the meta algorithm, even with this direct implementation, impractical.
In future work, an extended precision version, along with a control of error tolerances, will allow the integration of rank-one update steps avoiding full factorizations.
Leveraging higher-precision arithmetic will also allow us to exploit the pivoting framework for feasible regions yielding denser vertices and numerically challenging sparse linear systems to solve, such as the instance solved the $K$-sparse polytope.

A typical case where the conditioning of the matrix could degrade quickly occurred with $K$-sparse polytope instances with large enough $K$ and $\tau$ values. One vertex $\vv$ and its opposite $\vv_n = -\vv$ could be added to the active set, resulting in two columns such that:
\begin{align*}
	\frac{\langle\Tilde{\vv},\Tilde{\vv}_n\rangle}{\norm{\vv}^2} = \frac{1 -K\tau^2}{1 + K\tau^2} = -1 + \frac{2}{K\tau^2 + 1} \approx -1.
\end{align*}

As more almost-colinear pairs of columns are added, the matrix $M$ is closer to becoming singular, which may lead to imprecise or impossible linear system solves. The degradation in numerical accuracy is especially pronounced on problems yielding denser vertices.

In contrast, the operations performed on the active set weights in \fw{} variants are more robust to numerical accuracy, consisting only of increasing and decreasing individual weights or scaling the whole weight vector. The availability of an extended-precision sparse linear solver supporting rank-one updates will directly benefit implementations of our algorithms.

We also highlight that a key characteristic of \pm{} is that it maintains only the vertices required to represent the current iterate. One could solve the LPs over the final vertices to obtain a basic solution ensuring the Carathéodory upper bound on the cardinality, but it would require a constraint matrix including all vertices of the final active set, which can be arbitrarily larger than $n+2$ during the solving process.

Finally, we emphasize that \pm{} works with all algorithms leveraging convex combinations of vertices as long as their convergence respects the convex-combination-agnostic property introduced in Definition~\ref{def:cca}.
A promising avenue for future research will be leveraging the pivoting framework introduced in this paper in the context of stochastic gradient estimators, see \citet[Chapter 4]{braun2022conditional}.
The convergence of the stochastic Frank-Wolfe with and without momentum presented in the survey is convex combination agnostic, see \citet[Theorems 4.7, 4.8]{braun2022conditional}.
While large-scale settings could be detrimental to storing the vertices as an active set, stochastic variants of \fw{} algorithms can be leveraged for machine learning problems where the objective is a finite sum of loss terms evaluated over a large number of samples. In such a setting, the pivoting framework can maintain a sparse vertex decomposition of the iterates while maintaining the convergence guarantees of the underlying algorithm.
Away and Pairwise \fw{} variants have been proposed in the stochastic setting in \citet{goldfarb2017linear}. We note that the convergence rate is not convex-combination agnostic, since the number of vertices in the active set appears in a coefficient of the expected primal bound decrease rate, leading the authors to motivate a Carathéodory procedure after the active set update to bound the number of vertices, and thus improving their convergence rate, making a strong argument for the pivoting framework we propose.

\section{Numerical experiments (extended)}\label{sec:numerical_experiments}
We assess our algorithm on efficiency in terms of function value
and FW gap convergence, and sparsity of the obtained solutions, compared to the standard and lazified versions
(that is, using the active set vertices as a weak separation oracle \citep{braun2017lazifying})
of \afw{} and \bpfw{}, which also produce notably sparse solutions.
The implementation of further \pm variants is left for future research.
Our algorithm is implemented in Julia \citep{bezanson2017julia} v1.9.2 and builds on the \texttt{FrankWolfe.jl} package \citep{besanccon2022frankwolfe}.
The code for the algorithm and experiments is available on the repository \href{https://github.com/ZIB-IOL/pivoting_frankwolfe}{github.com/ZIB-IOL/pivoting\_frankwolfe}.

The sparse linear systems are solved with the LU decomposition of the \texttt{UMFPACK} library \citep{umfpack}.
All computations are executed on a cluster with PowerEdge R650 nodes equipped with Intel Xeon Gold 6342 2.80GHx CPUs and 512GB RAM in exclusive mode.
Plots are log-linear, function values are shifted so that the minimum on each plot reaches $10^{-8}$. The prefix \texttt{L-} is used to denote the lazified version of an algorithm, the prefix \texttt{P-} for the \pm variant.

\subsection{Sparse logistic regression}

\begin{figure}[t]
\centering
\begin{subfigure}[t]{0.9\textwidth}
    \centering
    \includegraphics[width=0.65\textwidth]{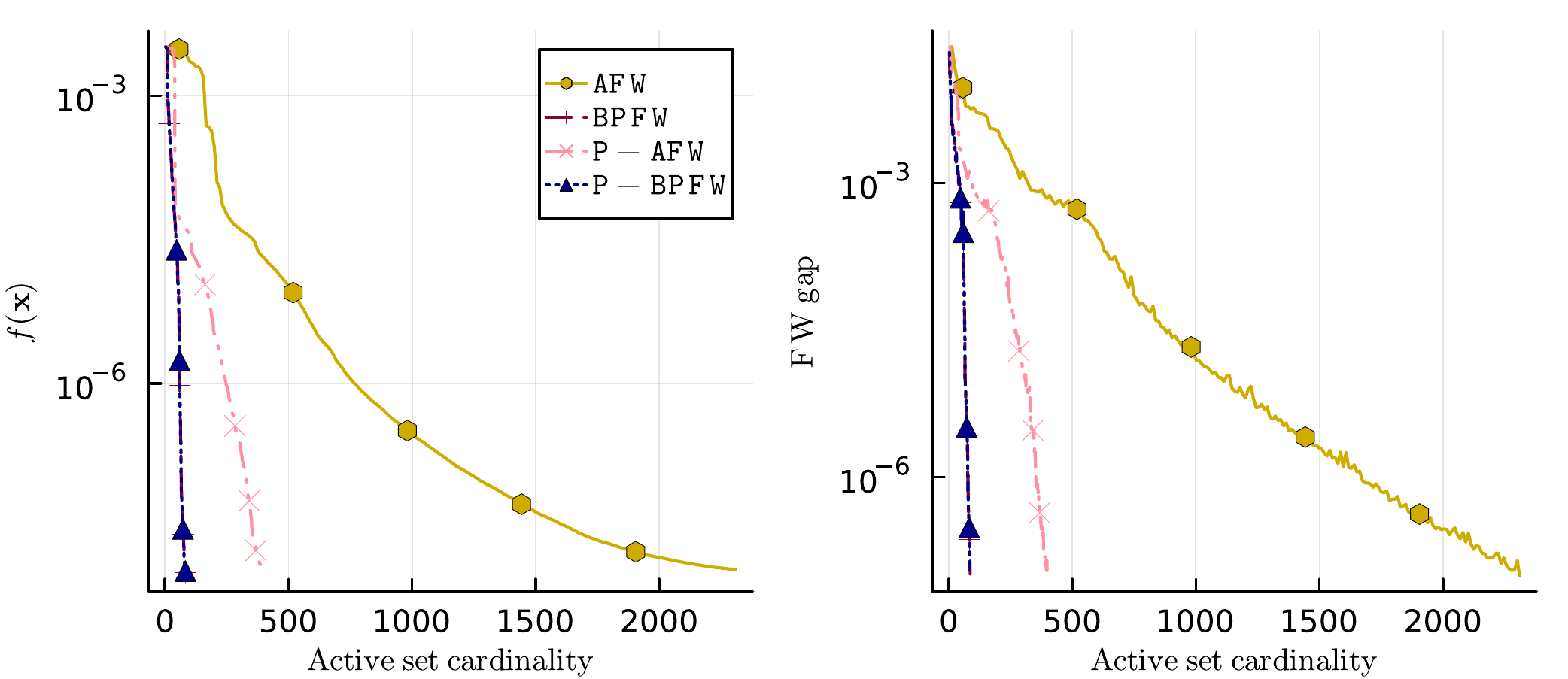}
\end{subfigure}
\hspace{0.05\textwidth}
\begin{subfigure}[t]{0.9\textwidth}
    \centering
    \includegraphics[width=0.65\textwidth]{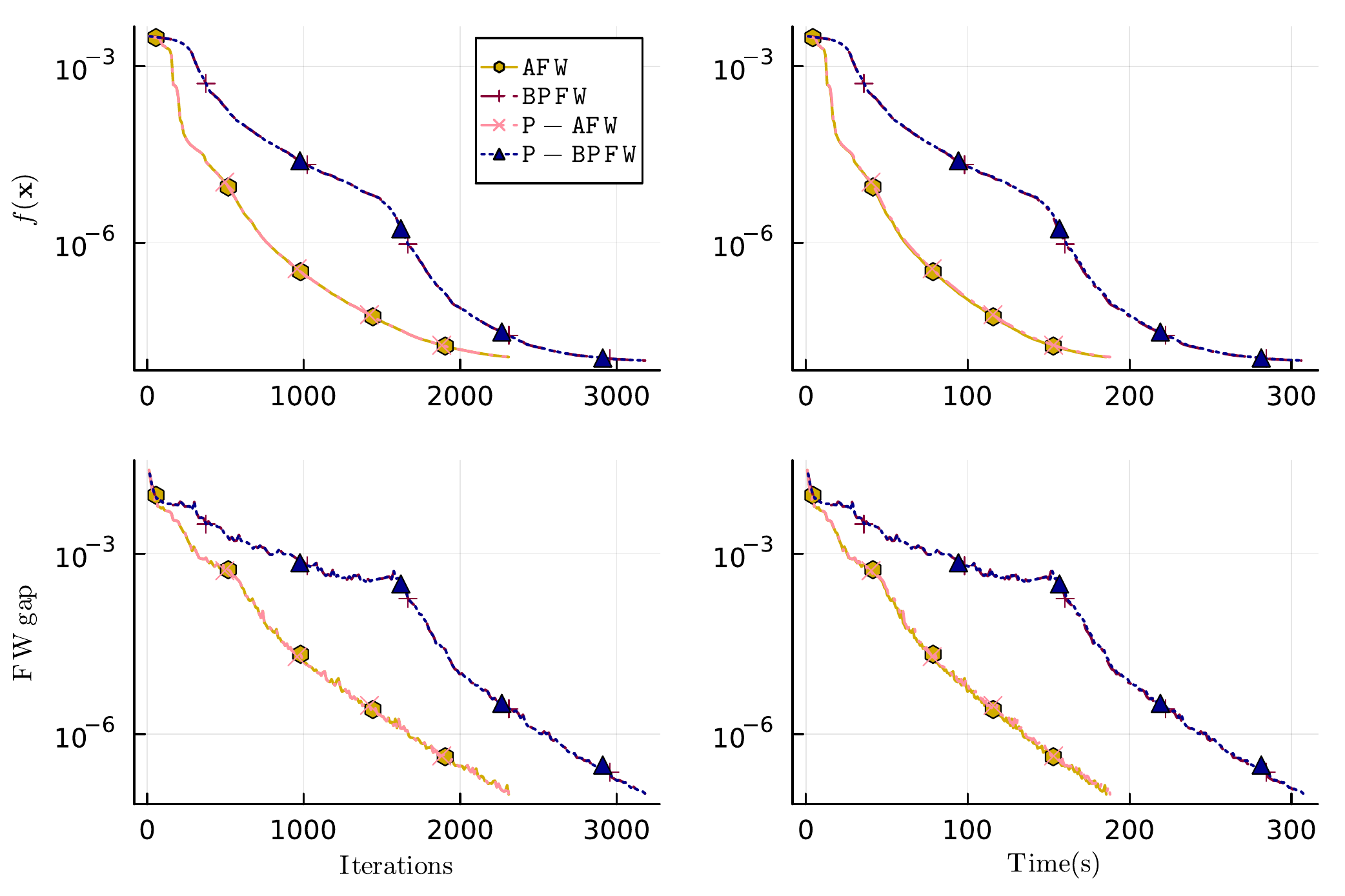}
\end{subfigure}
\caption{Logistic regression, $\tau=60$.}
\label{fig:logreg-60}
\end{figure}

\begin{figure}[t]
\centering
\begin{subfigure}[t]{0.9\textwidth}
    \centering
    \includegraphics[width=0.65\textwidth]{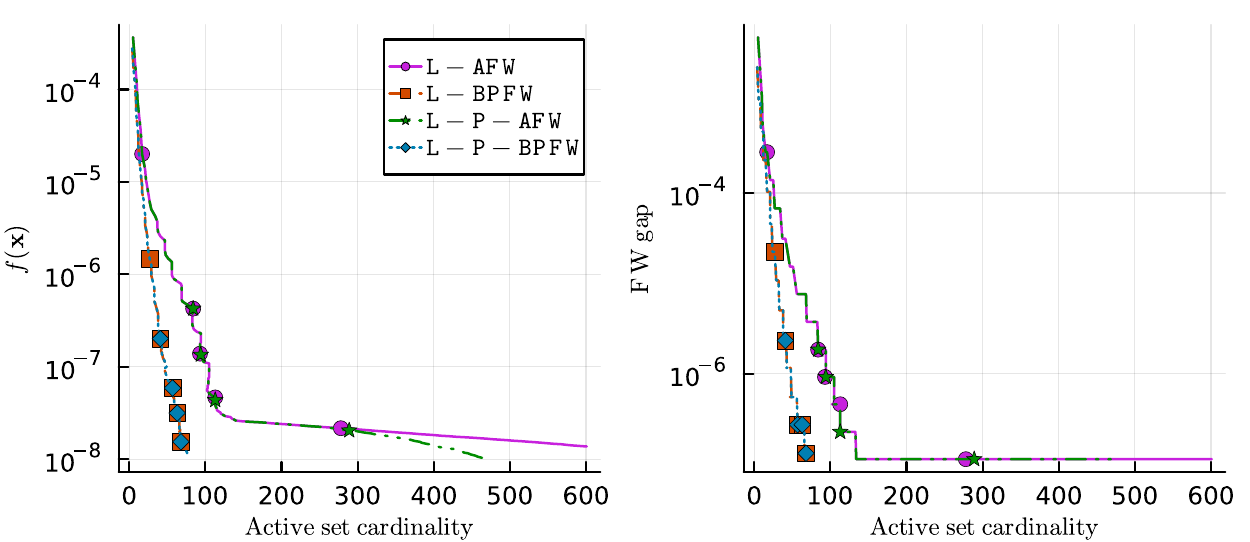}
\end{subfigure}
\hspace{0.05\textwidth}
\begin{subfigure}[t]{0.9\textwidth}
    \centering
    \includegraphics[width=0.65\textwidth]{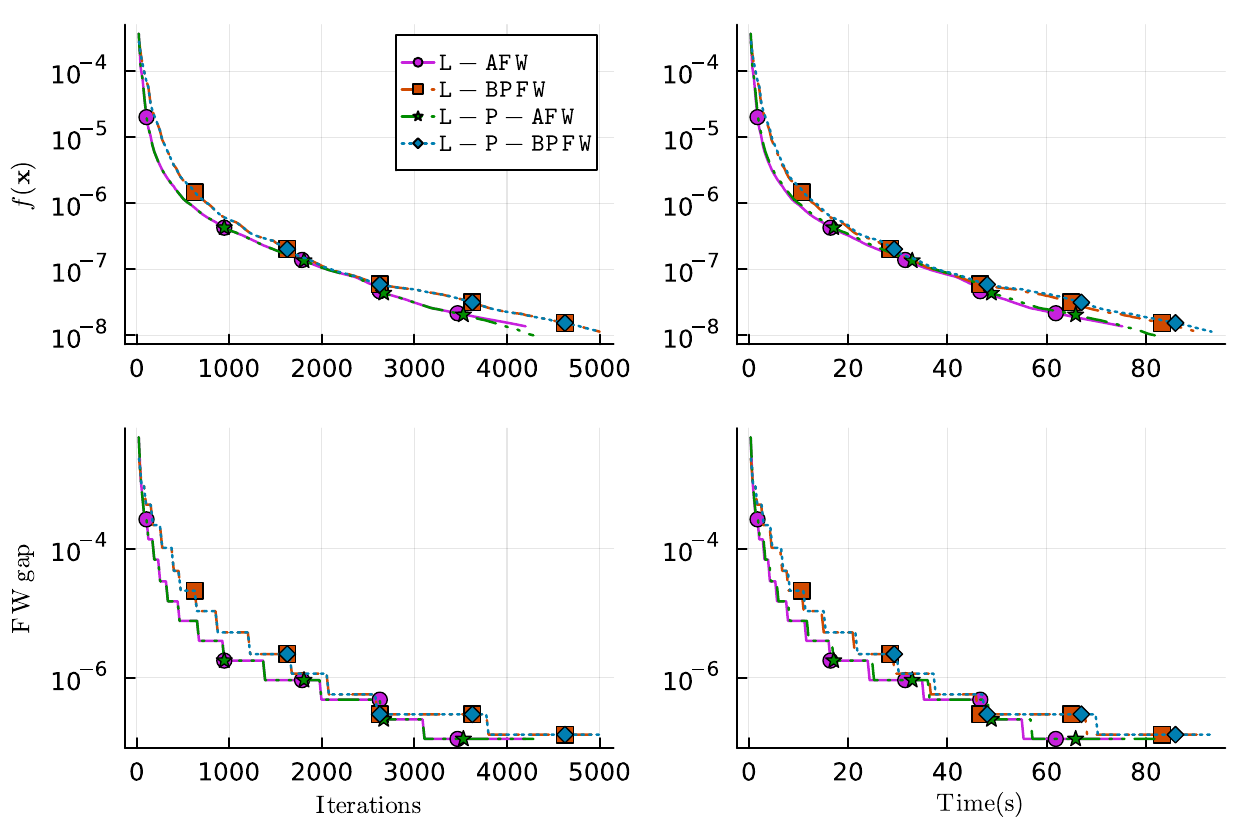}
\end{subfigure}
\caption{$K$-sparse logistic regression, $K=10$ and $\tau=1.0$.}
\label{fig:logreg-ksparse10}
\end{figure}

\begin{figure}[t]
\centering
\begin{subfigure}[t]{0.9\textwidth}
    \centering
    \includegraphics[width=0.65\textwidth]{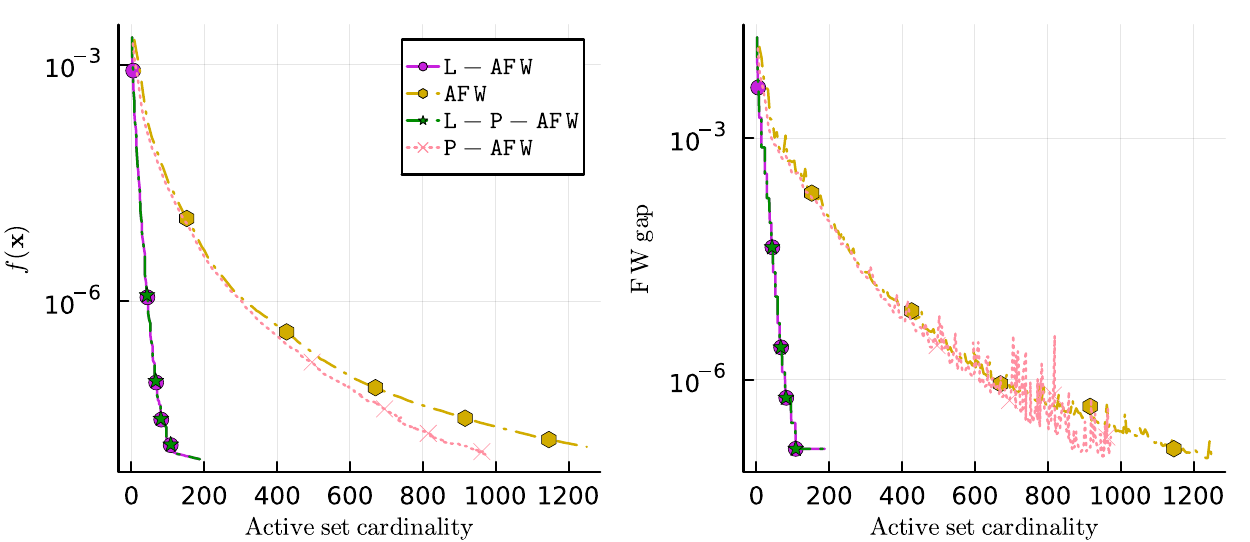}
\end{subfigure}
\hspace{0.05\textwidth}
\begin{subfigure}[t]{0.9\textwidth}
    \centering
    \includegraphics[width=0.65\textwidth]{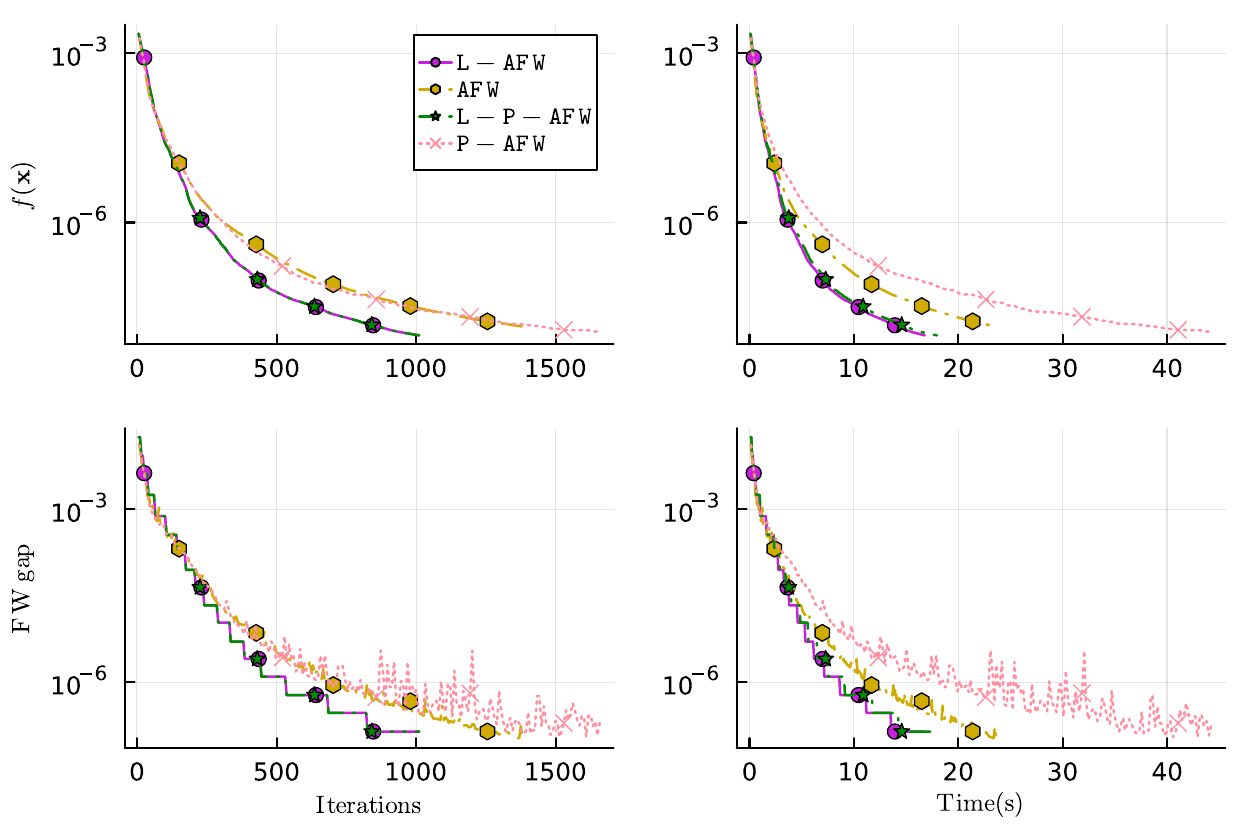}
\end{subfigure}
\caption{$K$-sparse logistic regression on the lazified and non-lazified \afw, $K=10$ and $\tau=4$.}
\label{fig:logreg-ksparse40}
\end{figure}

We run all algorithms on a logistic regression problem with an $\ell_1$-norm ball constraint:
\begin{align*}
    \min_{\xx} \; & \frac1m \sum_{i=1}^{m} \log(1+\exp (-y_i \mathbf{a}_i^\intercal \xx) \\
    \text{s.t.} \;& \|\xx\|_1 \leq \tau,
\end{align*}
where $m$ is the number of samples, $\tau > 0$ is the $\ell_1$-norm ball radius, $y_i \in \{-1,1\}$ encodes the class and $\mathbf{a}_i$ the feature vector for the $i$th sample.
We use the Gisette dataset \citep{guyon2004result}, which contains 5000 features, we run logistic regression on the validation set containing only 1000 samples and thus more prone to overfitting without a sparsity-inducing regularization constraint.
The convergence of the pivoting variants of both \afw{} and \bpfw{} converge similarly in both function value and FW gap as their standard counterparts as shown in
Figure~\ref{fig:logreg-60}.
Even though \bpfw{} typically maintains a smaller active set, it converges at a slower rate than the away-step \fw{} variants, both in function value and FW gap.
\texttt{P}-\afw{} drastically improves the sparsity of the \afw{} iterates, while maintaining the same convergence in function value and FW gap.
This highlights one key property of our meta algorithm: it can be adapted to several \fw{} variants, benefiting from their convergence rate while improving their sparsity.

We also fit logistic regression models on the $K$-sparse polytope \citep{cai2013sparse}:
\begin{align*}
    \mathcal{K}_{K,\tau} = & \{ \xx \in \R \mid \|\xx\|_1 \leq K \tau, \|\xx\|_{\infty} \leq \tau \} = \conv(\{ \xx \in \R \mid \|\xx\|_0 \leq K, |x_i| \in \{0,\tau\} \} ),
\end{align*}
the convex hull of vectors with at most $K$ non-zero entries with values in $\{-\tau,\tau\}$ for a given radius $\tau > 0$.

The relative behavior of the lazified version of the algorithms is illustrated on the K-sparse polytope in Figure~\ref{fig:logreg-ksparse10}. The lazified \afw variants converge faster in both function value and FW gap than the \bpfw variants, although the latter maintain much sparser iterates. \texttt{L-P-}\afw{} reduces the cardinality of the active set, especially when approaching the optimum.

We also study and compare the effect of the pivoting meta algorithm on the lazified and non-lazified version of \afw, illustrated on another regression on the $K$-sparse polytope in Figure~\ref{fig:logreg-ksparse40}.
On this instance, the pivoting technique applied to the lazified \afw algorithm only slightly improves the active set cardinality and does not change convergence. Pivoting however sharply reduces the cardinality of the active set for the non-lazified variant.
This however comes at the expense of a higher per-iteration runtime and higher numerical instabilities which affect in particular the FW gap convergence.
Such artifacts occur when the determinant of the matrix $M$ increases with new vertices, and the phenomenon is particularly pronounced for non-lazified variants which typically introduce more vertices. We discuss numerical robustness and stability in more detail in Subsection~\ref{sec:algorithmic}.

\subsection{Sparse signal recovery}

We assess our algorithm on a sparse signal recovery problem:
\begin{align*}
    \min_{\xx}\, &  \|A\xx - \yy\|_2^2  \\
    \text{s.t.}\; & \|\xx\|_1 \leq \tau,
\end{align*}
with $A\in\mathbb{R}^{m\times n}$, $\yy\in\R^m$, $ n > m$.
We generate the entries of the sensing matrix $A$ i.i.d.~from a standard Gaussian distribution and $\yy$ by adding Gaussian noise with unit standard deviation to $A \xx_{\text{true}}$,
with $\xx_{\text{true}}$ an underlying sparse vector, with $30\%$ of non-zero terms, all taking entries sampled from a standard Gaussian distribution.
The radius $\tau$ is computed as $\tau = \|\xx_{\text{true}}\|_1 / \tau_f$ for different values of $\tau_f$.

Figure~\ref{fig:signal-20} illustrates the results of the non-lazified version of \bpfw{} and \afw{} and their pivoting counterparts.
\texttt{P-\afw{}} converges at the same rate as \afw in function value and FW gap while being faster than both \bpfw{} variants, terminating before them, while maintaining an active set twice as sparse as \afw.

\begin{figure}[ht]
\centering
\begin{subfigure}[t]{0.9\textwidth}
    \centering
    \includegraphics[width=0.65\textwidth]{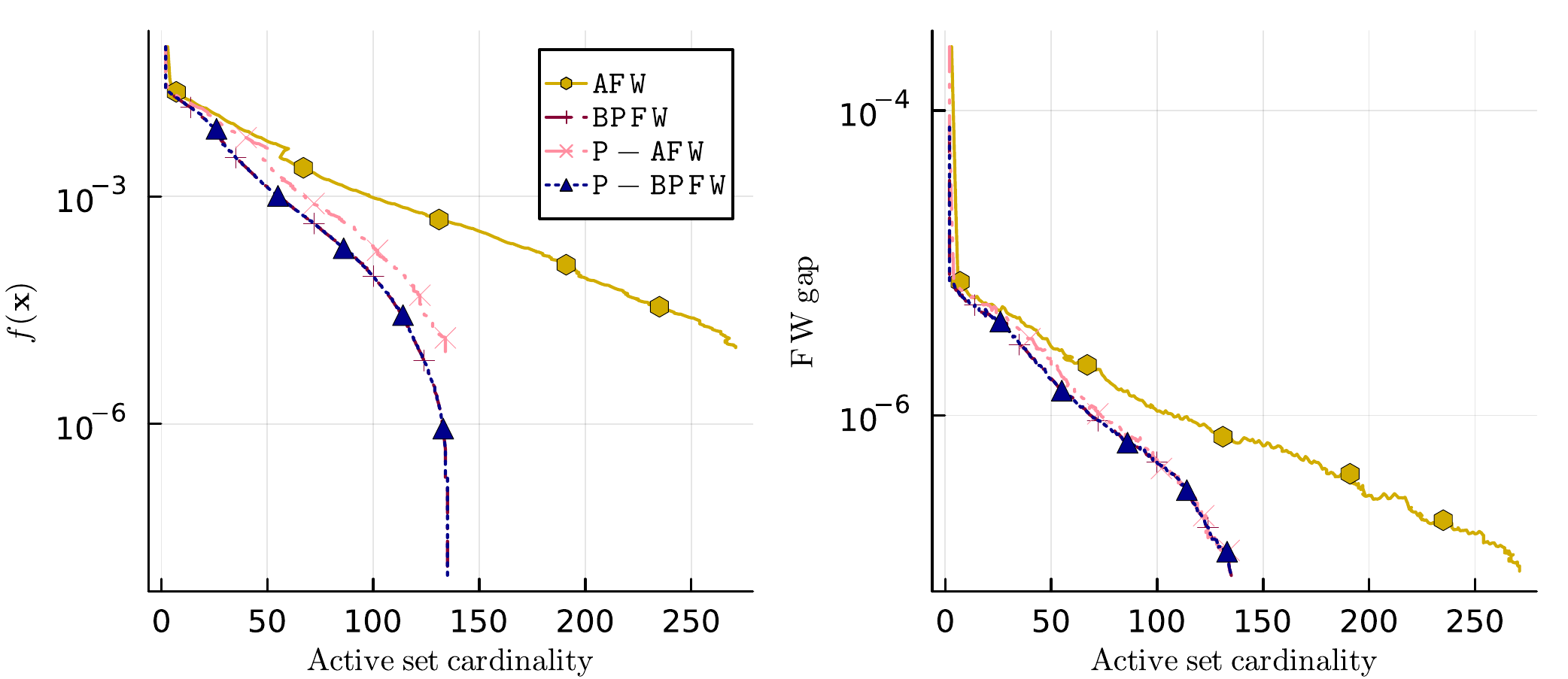}
\end{subfigure}
\hspace{0.05\textwidth}
\begin{subfigure}[t]{0.9\textwidth}
    \centering
    \includegraphics[width=0.65\textwidth]{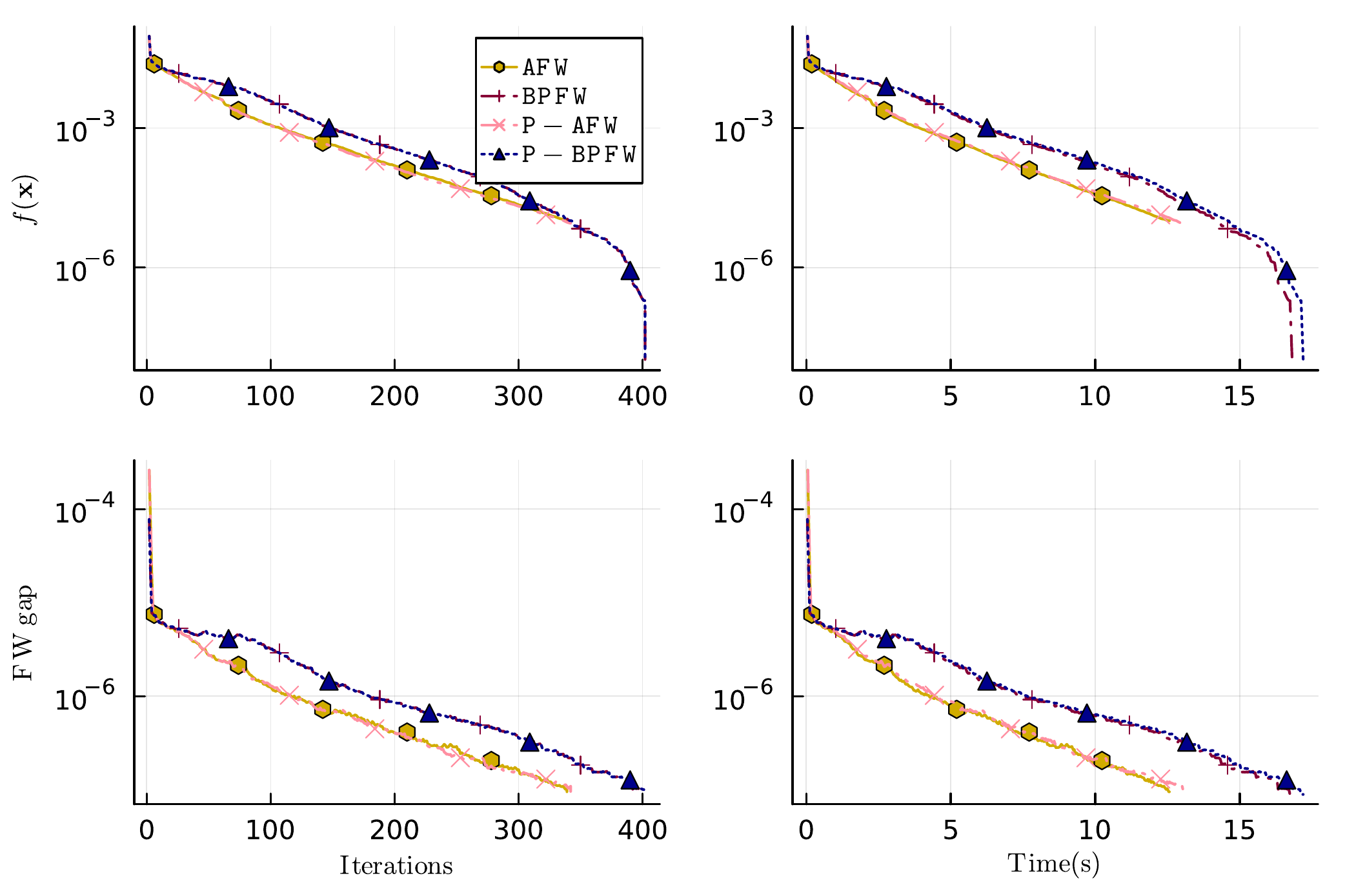}
\end{subfigure}
\caption{Signal recovery, $\tau_f=20$, $m = 6000$, $n = 14000$. All variants are non-lazified.}
\label{fig:signal-20}
\end{figure}

\end{document}